\documentclass[11pt,oneside,reqno]{amsart}
\usepackage{graphicx,color}
\usepackage{pict2e}
\usepackage{amssymb}
\usepackage{amsthm}                
\usepackage[margin=1in]{geometry}  
\usepackage{graphics}
\usepackage{epstopdf}
\usepackage{amsmath}
\usepackage{graphicx}
\usepackage{subfigure}
\usepackage{latexsym}
\usepackage{amsmath}
\usepackage{amsfonts}
\usepackage{amssymb}
\usepackage{tikz}
\usepackage{mathdots}
\usepackage{comment}
\usepackage{float}
\usepackage[mathscr]{euscript}
\usepackage{enumerate}
\usepackage{epsfig}
\usepackage { graphicx }
\usepackage { hyperref }
\usepackage { graphics }
\usepackage { graphicx }
\usepackage{xparse}
\usepackage{epstopdf}
\usepackage {eufrak}
\usepackage[utf8]{inputenc}
\usepackage{longtable}

\usepackage[lite]{amsrefs}

\theoremstyle{definition}
\newtheorem{theorem}{Theorem}[section]

\newtheorem{corollary}[theorem]{Corollary}

\theoremstyle{definition}
\newtheorem{definition}[theorem]{Definition}
\newtheorem{example}[theorem]{Example}
\theoremstyle{remark}

\numberwithin{equation}{section}

\numberwithin{equation}{section}

\begin{document}

\title{Foundations of the Colored Jones Polynomial of singular knots}

\author{Mohamed Elhamdadi}
\address{Department of Mathematics, University of South Florida, 
Tampa, FL 33647 USA}
\email{emohamed@math.usf.edu }

\author{Mustafa Hajij}
\address{Department of Mathematics, University of South Florida, 
Tampa, FL 33647 USA}
\email{mhajij@usf.edu}

\begin{abstract}
This article gives the foundations of the colored Jones polynomial for singular knots. We extend Masbum and Vogel's algorithm \cite{MV} to compute the colored Jones polynomial for any singular knot. We also introduce the tail of the colored Jones polynomial of singular knots and use its stability properties to prove a false theta function identity that goes back to Ramanujan.
    
\end{abstract}

 \maketitle

 \tableofcontents
\section{Introduction}
The colored Jones polynomial $J_{n,L}(q)$ of a link $L$ is a sequence of Laurent polynomials in the variable $q$. The label $n$ is a positive integer which usually stands for the color. The study of the Jones polynomials and, in general, quantum invariants have attracted much attention in the past 30 years.  Particularly over the past decade, a growing interest has appeared in regards of the coefficients of the colored Jones polynomial. The interest stems mainly from certain stability behavior these coefficients have for adequate knots and links. The stability of the colored Jones polynomial was first observed by Dasbach and Lin in \cite{DL}, where they showed that for an alternating link $L$ the absolute values of the first and last three leading coefficients of $J_{n,L}(q)$ are independent of the color $n$ for sufficiently large values of $n$. This finding was used to derive an upper bound for the volume of the complement of alternating prime non-torus knots in terms of the leading two and last two coefficients of $J_{2,K}(q)$. In the same article, Dasbach and Lin conjectured that the first $n$ coefficients of $J_{n+1,L}(q)$ agree with the last $n$ coefficients of $J_{n,L}(q)$ for any alternating link $L$. Thus, the tail of the colored Jones polynomial for such a stable sequence is a $q-$series whose first $n$ coefficients agree with the first $n$ coefficients of $J_{n,L}(q)$. This stability was then proven for adequate knots by Armond in \cite{Armond}. Furthermore, Garoufalidis and L\^e \cite{GL} gave another proof for alternating knots using different techniques. They also proved that higher order stability occurs for alternating knots. Lee \cite{lee2016trivial} extended Armond's result to all links and showed that the tail of a link $L$ is trivial if and only if $L$ is non A-adequate. Recently, Lee gave categorified version of her result in \cite{lee2016trivial} proving a conjecture of Rozansky \cite{rozansky2012khovanov} stating that the categorification of the colored Jones polynomial of a non A-adequate link has a trivial tail homology. Other work on the stability of the colored Jones polynomial can be found in \cite{Hajij3} where the stability is shown using simple skein theoretic techniques. The work of Armond and Dasbach was then extended to the quantum spin network in \cite{Hajij1} and \cite{Hajij2}. 

One of the primary interests of these coefficients seems to be driven from their relation to the famous Ramanujan-type $q$-series. 
One of the earliest connections with the Ramanujan type $q$-series was observed in \cite{Hikami}, in which the author studied the asymptotic behaviors of the colored Jones polynomials. The $q$-series associated with the colored Jones polynomial exhibits many interesting properties. In fact, for many knots with small crossings, these $q$-series are equal to theta functions or false theta functions. More interestingly, the study of the tail has been used to prove Andrews-Gordon identities for the two-variable, Ramanujan theta function in \cite{CodyOliver} and corresponding identities for the false theta function in \cite{Hajij2}. These two families of $q$-series identities were obtained from investigating $(2,p)$-torus knots. For $q$-series techniques proving these identities, refer to \cite{Robert}. 
The colored Jones polynomial was extended to singular knots using skein theory in \cite{bataineh2016colored}. In this paper, extended Masbum and Vogel's algorithm \cite{MV} to compute the colored Jones polynomial for any singular link. Furthermore, we investigate the stability of these coefficients for certain singular torus knots and show how they can be used to prove natural Ramanujan-type identities.  

Since we want this article to be self contained we include reviews of the necessary material for the convenience of the reader.  The paper is organized as follows:  In section \ref{sec2} we give the basics of the Kauffman bracket skein module and the colored Jones polynomial. In section \ref{sec4} we give the basics of the colored Jones polynomial for singular links and extend Masbum and Vogel's algorithm \cite{MV} to compute the colored Jones polynomial for singular links. In Section \ref{sec5} we compute the tail of the colored Jones polynomial of singular torus knots and we show that this tail gives a natural Ramanujan-type $q$-series identity.

\section{The Kauffman Bracket Skein Module}\label{sec2}

Let $M$ be an oriented 3-manifold. A framed link in $M$ is an oriented embedding of a disjoint union of oriented annuli in $M$.  A framed point in the boundary $\partial M$ of $M$ is a closed interval in $\partial M$. A \textit{band} in $M$ is an oriented embedding of $I \times I$ into $M$ that meets $\partial M$ orthogonally at two framed points $x$ and $y$ in $\partial M$.

\begin{definition} \cite{Przytycki}
Let $M$ be an oriented $3$-manifold and let $A$ be an invertible element in a commutative ring $\mathcal{R}$ with a unit. Let $\mathcal{L}_M$ denotes the set of all isotopy classes of unoriented framed links in $M$. The empty link is considered to be an element of $\mathcal{L}_M$. Let $\mathcal{R}\mathcal{L}_M$ be the free $\mathcal{R}$-module generated by $\mathcal{L}_M$. Let $R(M)$ is the submodule of $\mathcal{R}\mathcal{L}_M$ generated by all expressions of the form
\begin{eqnarray*}(1)\hspace{3 mm}
  \begin{minipage}[h]{0.06\linewidth}
        \vspace{0pt}
        \scalebox{0.04}{\includegraphics{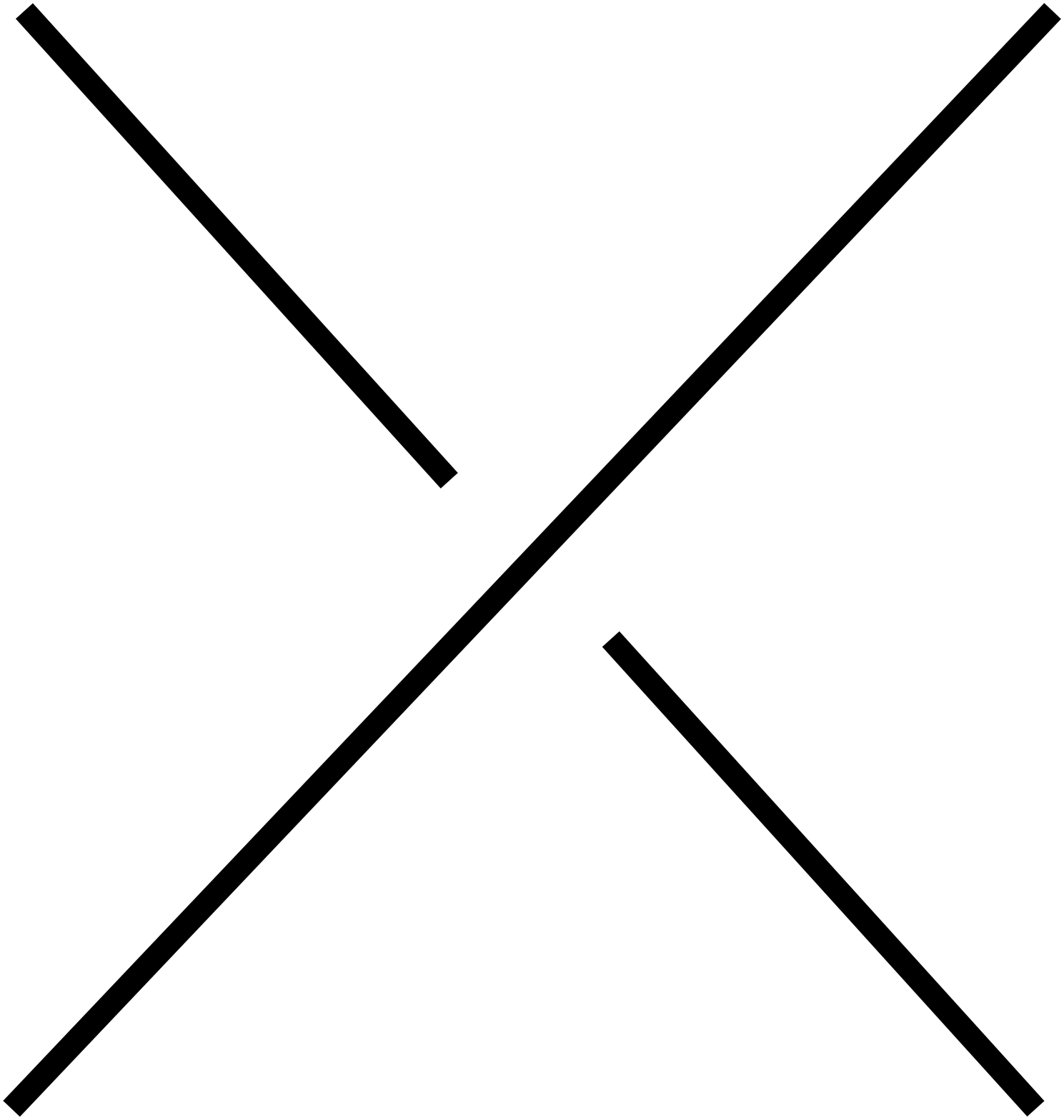}}
   \end{minipage}
   -
     A 
  \begin{minipage}[h]{0.06\linewidth}
        \vspace{0pt}
        \scalebox{0.04}{\includegraphics{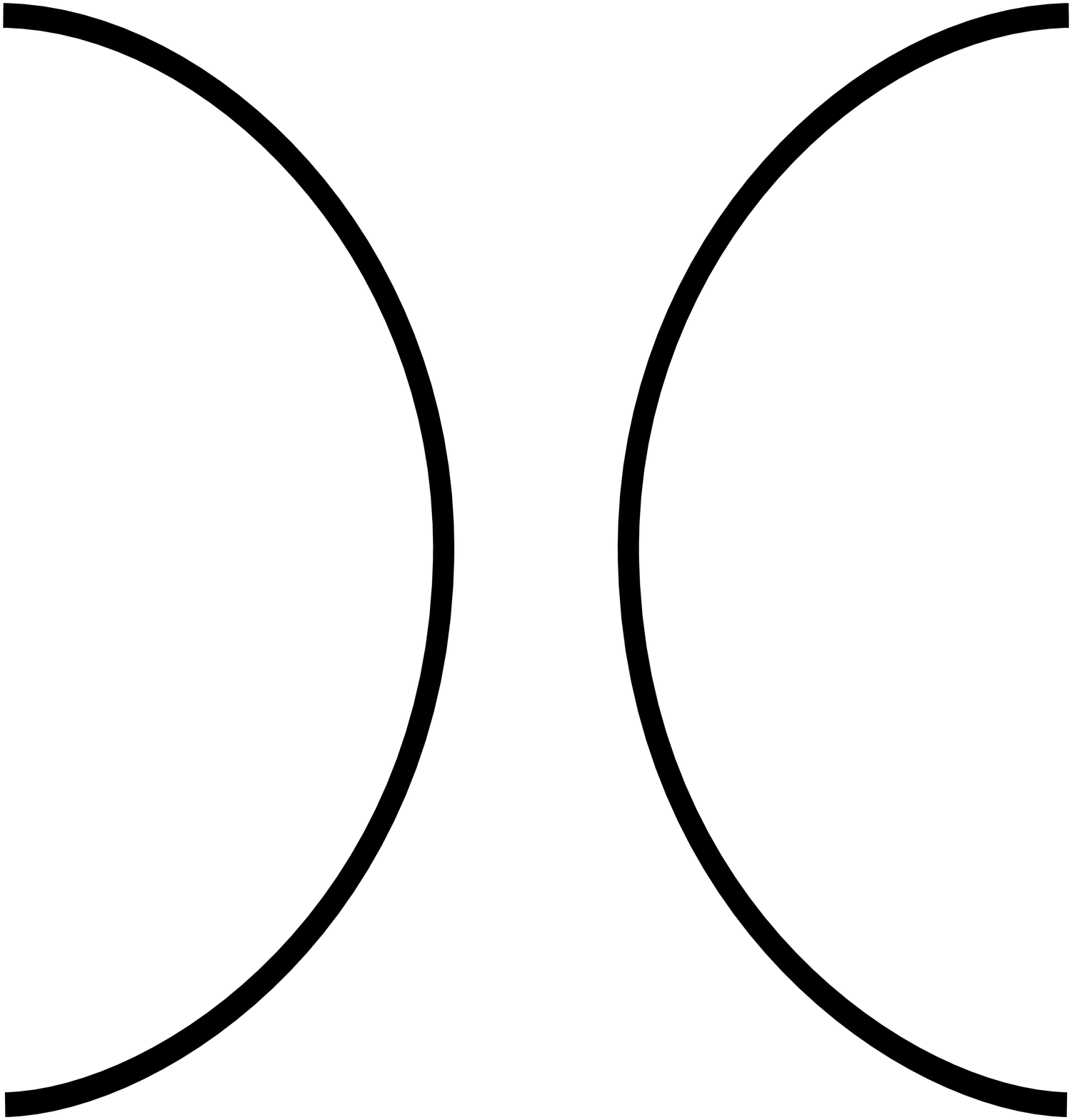}}
   \end{minipage}
   -
  A^{-1} 
  \begin{minipage}[h]{0.06\linewidth}
        \vspace{0pt}
        \scalebox{0.04}{\includegraphics{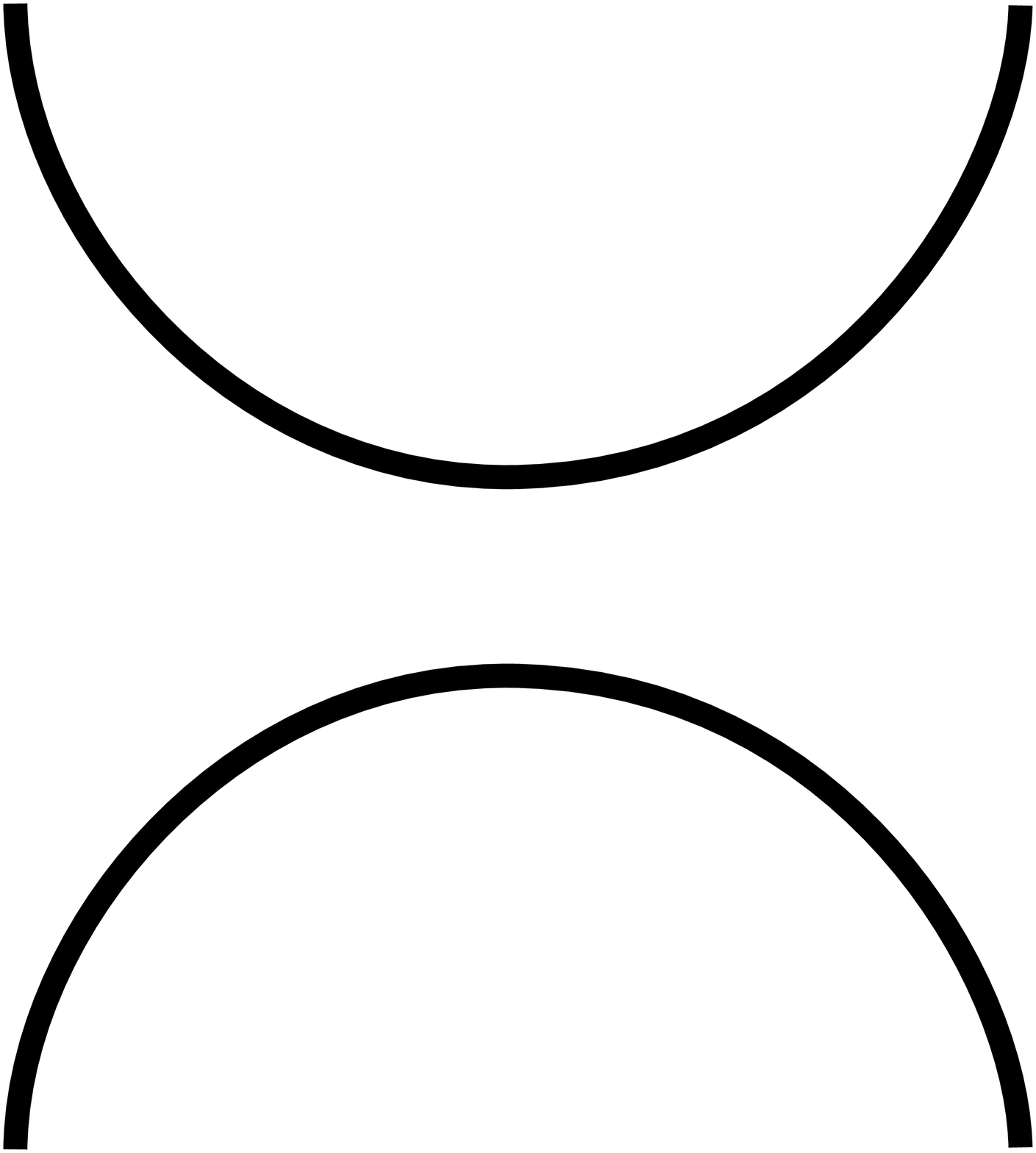}}
   \end{minipage}
, \hspace{20 mm}
  (2)\hspace{3 mm} L\sqcup
   \begin{minipage}[h]{0.05\linewidth}
        \vspace{0pt}
        \scalebox{0.02}{\includegraphics{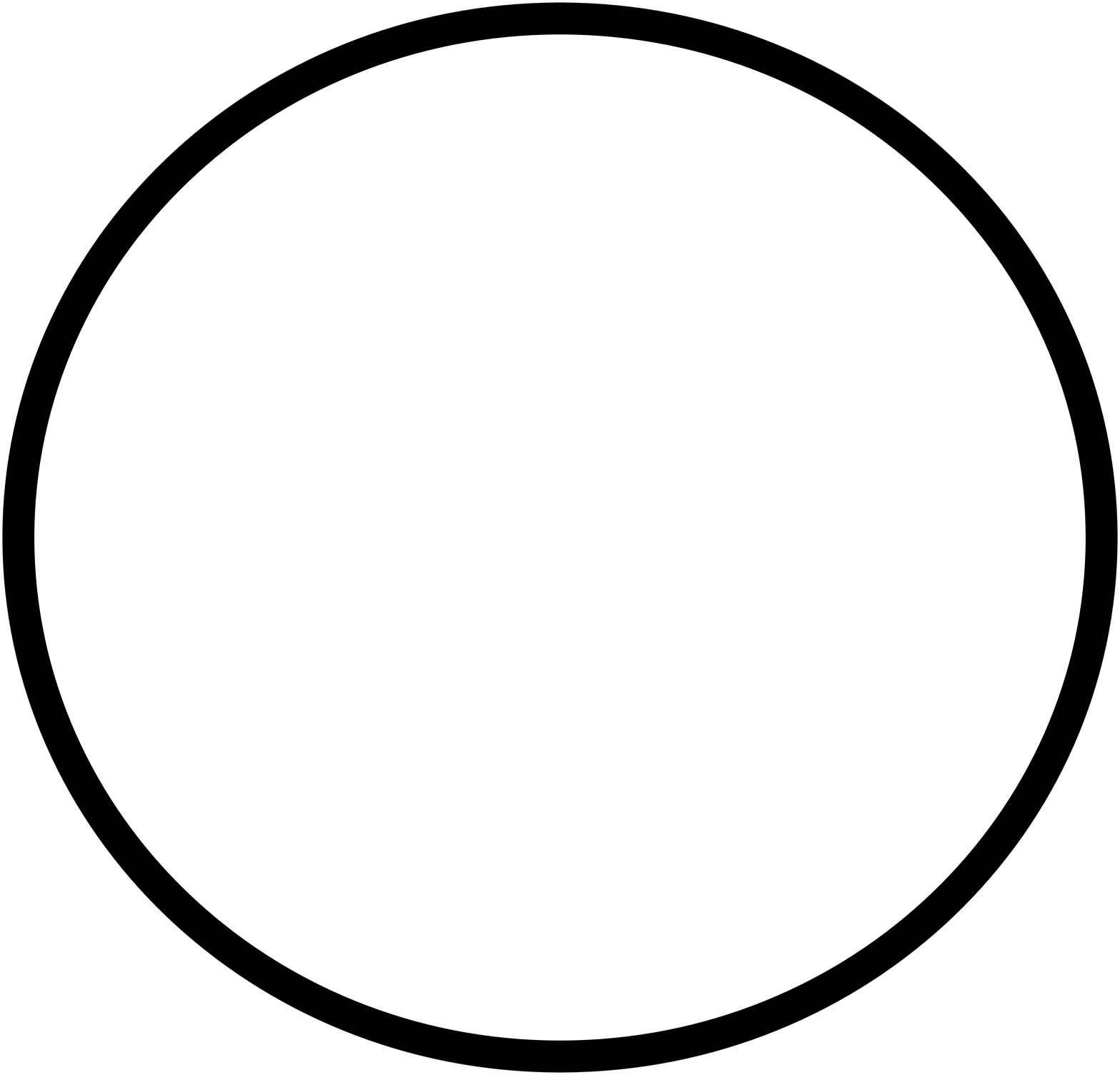}}
   \end{minipage}
  +
  (A^{2}+A^{-2})L, 
  \end{eqnarray*}
where $L\sqcup$ \begin{minipage}[h]{0.05\linewidth}
        \vspace{0pt}
        \scalebox{0.02}{\includegraphics{simple-circle}}
   \end{minipage}  is the disjoint union of a framed link $L$ in $M$ and the trivial framed knot 
   \begin{minipage}[h]{0.05\linewidth}
        \vspace{0pt}
        \scalebox{0.02}{\includegraphics{simple-circle}}
   \end{minipage}.
\textit{The Kauffman bracket skein module} of the $3$-manifold $M$ is the quotient module
\begin{equation}
\mathcal{S}(M,\mathcal{R},A)=\mathcal{R}\mathcal{L}_M/R(M),
\end{equation}

\end{definition}

When the context is clear, we will write $\mathcal{S}(M)$ instead of $\mathcal{S}(M,\mathcal{R},A)$.
The definition of the Kauffman bracket skein module can be extended to $3$-manifolds with boundaries. Let $x_1, \cdots , x_{2n}$ be a set, possibly empty, of designated framed points on $\partial M$. Let $\mathcal{L}_M$ be the set of all surfaces in $M$ decomposed into a union of finite number of framed links and bands joining the points $\{ x_i \}_{i=1}^{2n}$. The relative Kauffman bracket skein module is defined to be the quotient module \begin{equation}
\mathcal{S}(M,\mathcal{R},A,\{ x_i \}_{i=1}^{2n})=\mathcal{R}\mathcal{L}_M/R(M).
\end{equation}
Note that the construction of the relative Kauffman bracket skein module is functorial in the sense that an embedding of oriented $3$-manifolds with $2n$ (framed) points on the boundaries 
\begin{equation}
j: (M,\{ x_i \}_{i=1}^{2n}) \hookrightarrow (M',\{ y_i \}_{i=1}^{2n})
\end{equation}
 induces a homomorphism of $\mathcal{R}$-modules  \begin{equation}
\mathcal{S}(M,\mathcal{R},A,\{ x_i \}_{i=1}^{2n}) \rightarrow \mathcal{S}(M',\mathcal{R},A,\{ y_i \}_{i=1}^{2n}).  
\end{equation}
If the $3$-manifold $M$ is homeomorphic to $F\times I$ where $F$ an oriented surface with a finite set of points (possibly empty) in its boundary $\partial  F$ and $I$ is an interval, then one can project framed links in $M$ to link diagrams in $F$.\\
It is well known (\cite{Przytycki}) that the Kauffman bracket skein module of the $3$-sphere $S^3$ is free on the empty link, that is $\mathcal{S}(S^3)= \mathcal{R}$ . Now we consider the relative Kauffman bracket skein module of $D^3=I\times I \times I$ with $2n$ marked points on its boundary $\partial D^3$.  The first $n$ points are placed on the top edge of $D^3$ and the other $n$ points on the bottom edge. Recall that the relative skein module does not depend on the exact position of the points $\{x_i\}_{i=1}^{2n}$. However, one needs to specify the position here in order to define an algebra structure on  $\mathcal{S}(D^3,\mathcal{R},A,\{ x_i \}_{i=1}^{2n})$. 
 
 Let $S_1$ and $S_2$ be two elements in $\mathcal{L}_{D^3}$ such that $\partial S_j$, where $j=1,2$, consists of the points $\{x_i\}_{i=1}^{2n}$ that we specified above. 
 
 Define $S_1 \times S_2$ to be the surface in $D^3$ obtained by attaching $S_1$ on the top of $S_2$ and then compress the result to $D^3$. This multiplication extends to a well-defined multiplication on $\mathcal{S}(D^3,\mathcal{R},A,\{ x_i \}_{i=1}^{2n})$. With this multiplication the module $\mathcal{S}(D^3,\mathcal{R},A,\{ x_i \}_{i=1}^{2n})$ becomes an associative algebra over $\mathcal{R}$ known as the \textit{$n^{th}$ Temperley-Lieb algebra} $TL_n$. 
The ring $\mathcal{R}$ will be the field $\mathbb{Q}(A)$ generated by the indeterminate $A$ over the rational numbers through the rest of the paper.

\subsection{The Jones-Wenzl Idempotents}
Our invariant is defined in terms of the Jones-Wenzl idempotent (JWI). This is an element in  $TL_n$ denoted by $f^{(n)}$ that has played a crucial role in the understanding of the Temperley-Lieb algebra and its applications. The idempotent has a central role in defining the $SU(2)$ Witten-Reshetikhin-Turaev Invariants  \cite{KauffLin,Lic92,RT}. It also has a major importance in the colored Jones polynomial and its applications \cite{BHMV,RT,Hajij2,Hajij1,TuraevViro92}, and quantum spin networks \cite{MV}. The definition of the projector goes back to Jones \cite{J2}. The recursive formula we will use here goes back to Wenzl \cite{Wenzl}: 
\begin{align}
\label{recursive}
  \begin{minipage}[h]{0.05\linewidth}
        \vspace{0pt}
        \scalebox{0.12}{\includegraphics{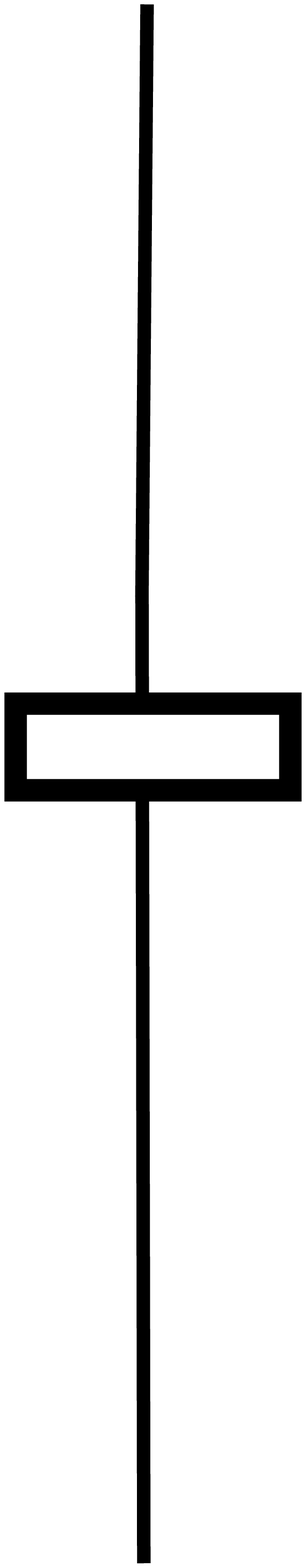}}
         \put(-20,+70){\footnotesize{$n$}}
   \end{minipage}
   =
  \begin{minipage}[h]{0.08\linewidth}
        \hspace{8pt}
        \scalebox{0.12}{\includegraphics{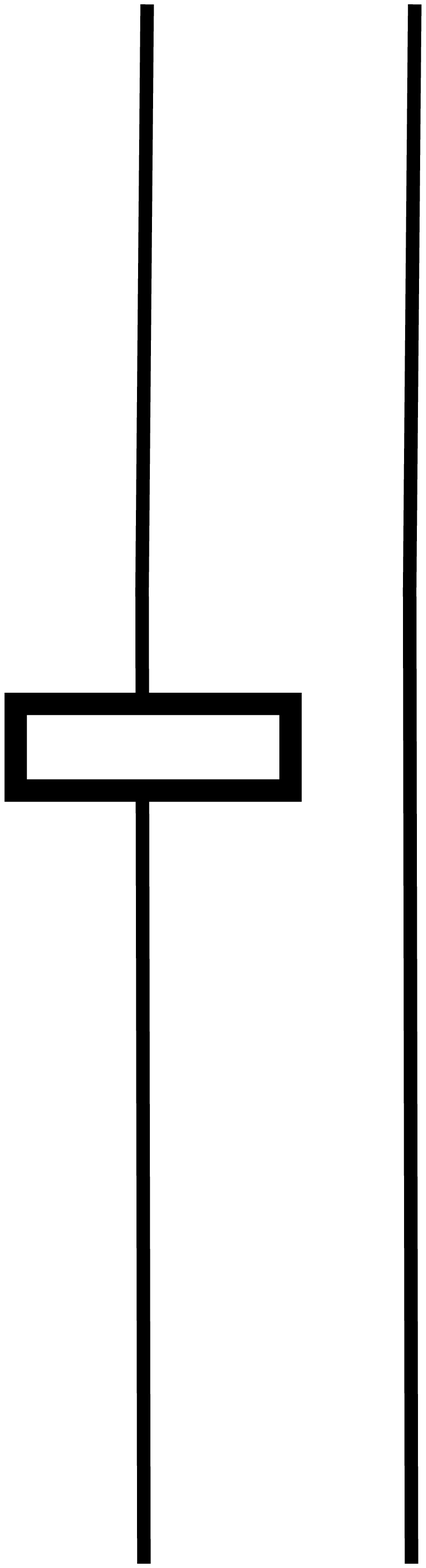}}
        \put(-42,+70){\footnotesize{$n-1$}}
        \put(-8,+70){\footnotesize{$1$}}
   \end{minipage}
   \hspace{9pt}
   -
 \Big( \frac{\Delta_{n-2}}{\Delta_{n-1}}\Big)
  \hspace{9pt}
  \begin{minipage}[h]{0.10\linewidth}
        \vspace{0pt}
        \scalebox{0.12}{\includegraphics{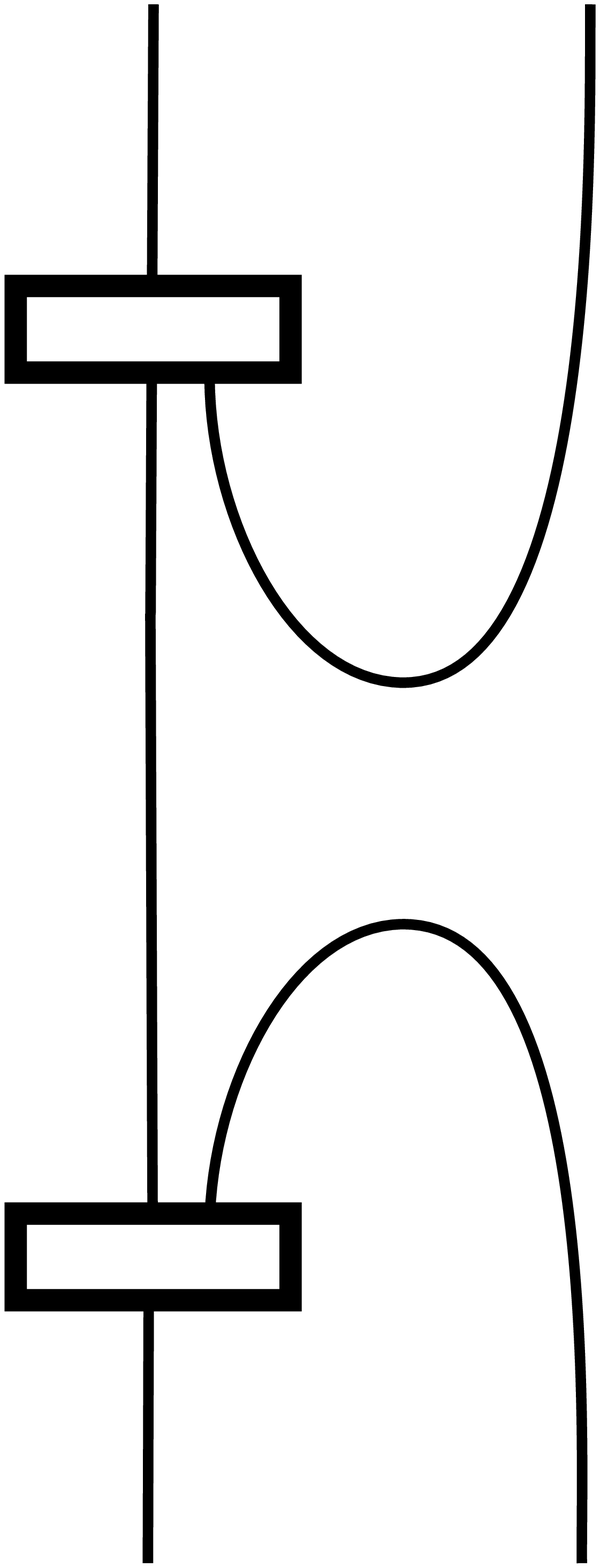}}
         \put(2,+85){\footnotesize{$1$}}
         \put(-52,+87){\footnotesize{$n-1$}}
         \put(-25,+47){\footnotesize{$n-2$}}
         \put(2,+10){\footnotesize{$1$}}
         \put(-52,+5){\footnotesize{$n-1$}}
   \end{minipage}
  , \hspace{20 mm}
    \begin{minipage}[h]{0.05\linewidth}
        \vspace{0pt}
        \scalebox{0.12}{\includegraphics{nth-jones-wenzl-projector}}
        \put(-20,+70){\footnotesize{$1$}}
   \end{minipage}
  =
  \begin{minipage}[h]{0.05\linewidth}
        \vspace{0pt}
        \scalebox{0.12}{\includegraphics{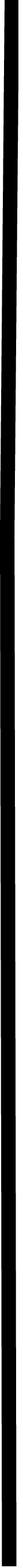}}
   \end{minipage}   
  \end{align}
  where 
\begin{equation*}
 \Delta_{n}=(-1)^{n}\frac{A^{2(n+1)}-A^{-2(n+1)}}{A^{2}-A^{-2}}.
\end{equation*} 

The idempotent satisfies the following characterizing properties: 

\begin{eqnarray}
\label{properties}
\hspace{0 mm}
    \begin{minipage}[h]{0.21\linewidth}
        \vspace{0pt}
        \scalebox{0.115}{\includegraphics{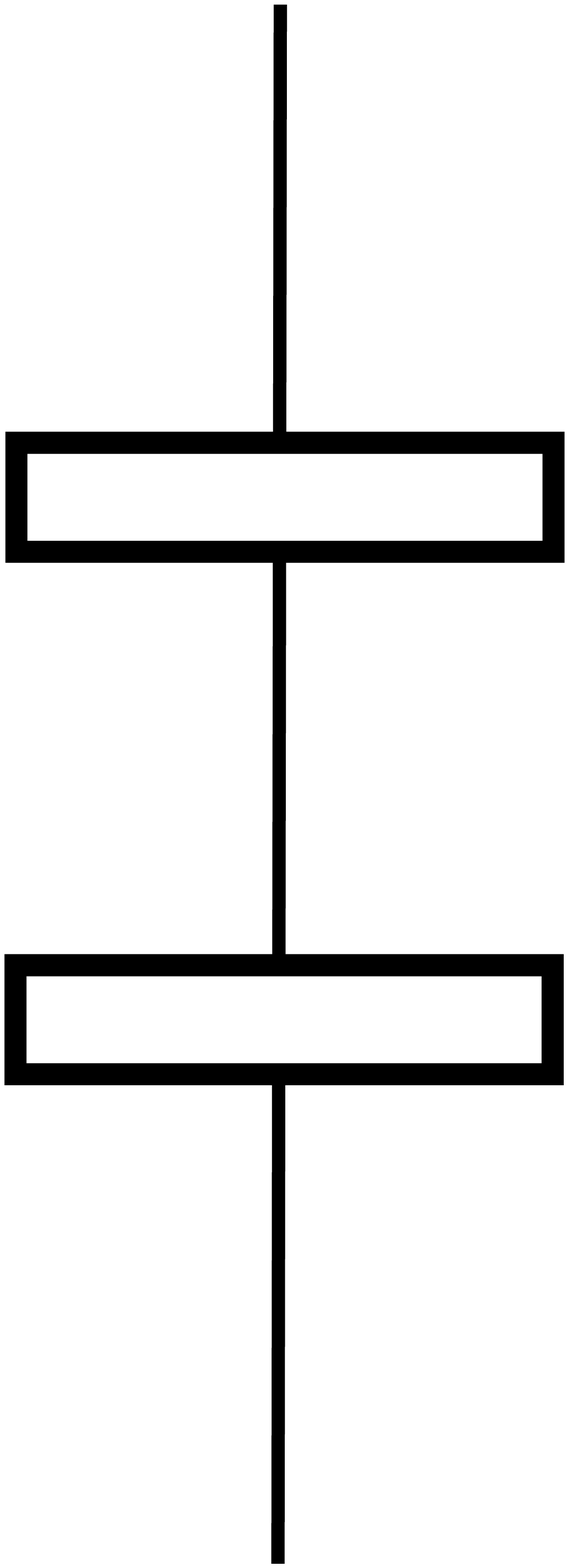}}
        \put(0,+80){\footnotesize{$n$}}
       
   \end{minipage}
  = \hspace{5pt}
     \begin{minipage}[h]{0.1\linewidth}
        \vspace{0pt}
         \hspace{50pt}
        \scalebox{0.115}{\includegraphics{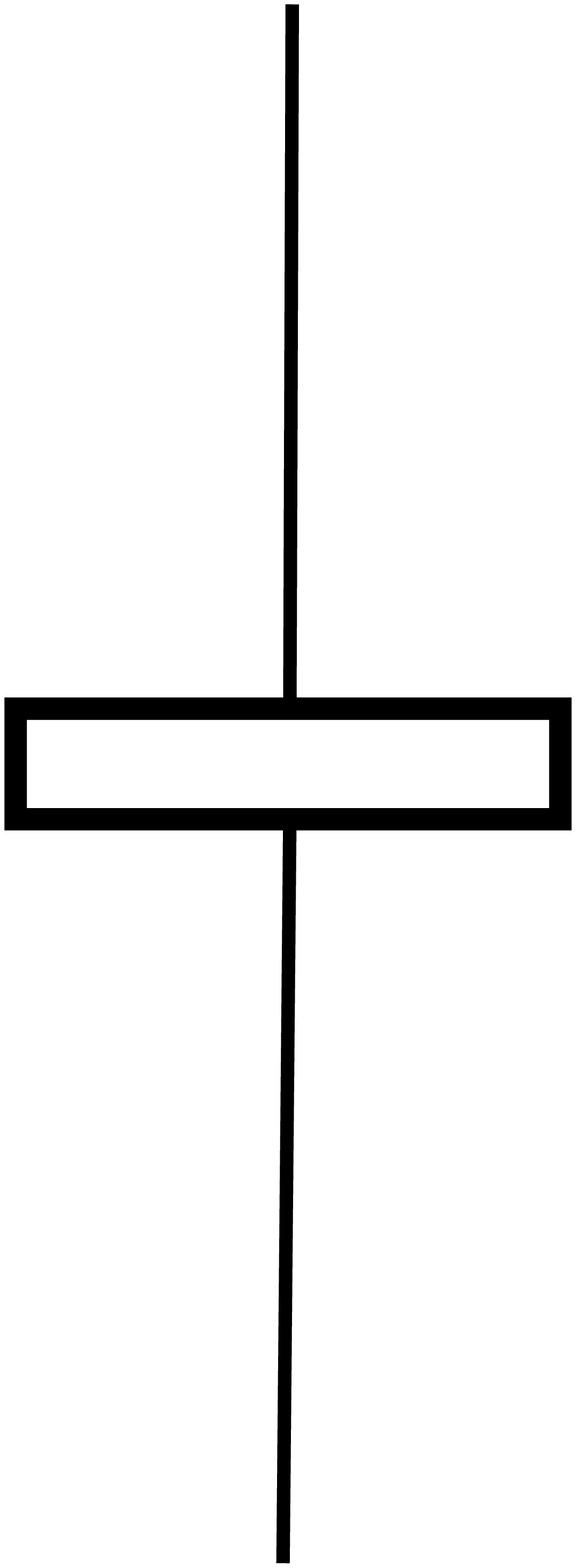}}
        \put(-60,80){\footnotesize{$n$}}
   \end{minipage}
    , \hspace{15 mm}
    \begin{minipage}[h]{0.09\linewidth}
        \vspace{0pt}
        \scalebox{0.115}{\includegraphics{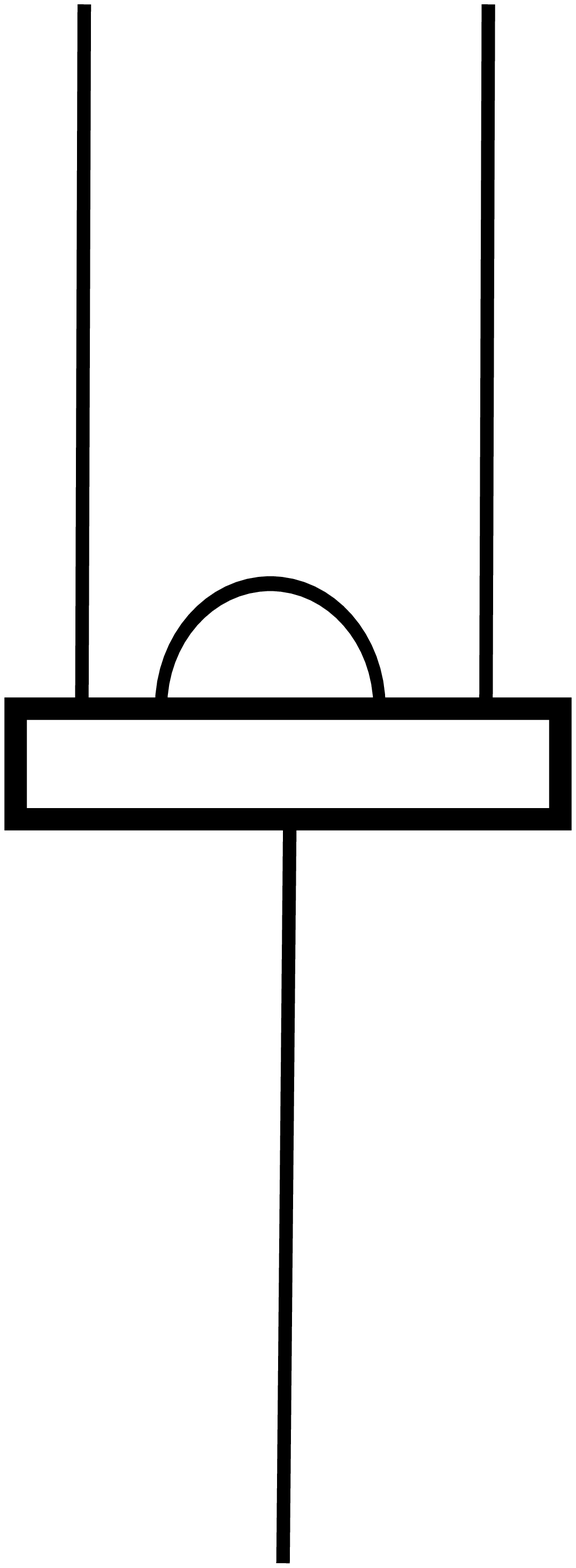}}
         \put(-70,+82){\footnotesize{$n-i-2$}}
         \put(-20,+64){\footnotesize{$1$}}
        \put(-2,+82){\footnotesize{$i$}}
        \put(-28,20){\footnotesize{$n$}}
   \end{minipage}
   =0,
   \label{AX}
  \end{eqnarray}
The first property is called the idempotency property of the JWI and the second property  is called the annihilation property. The JWI also satisfies the following:   
\begin{eqnarray}
\label{properties}
\hspace{0 mm}
     \begin{minipage}[h]{0.08\linewidth}
        \vspace{0pt}
        \scalebox{0.115}{\includegraphics{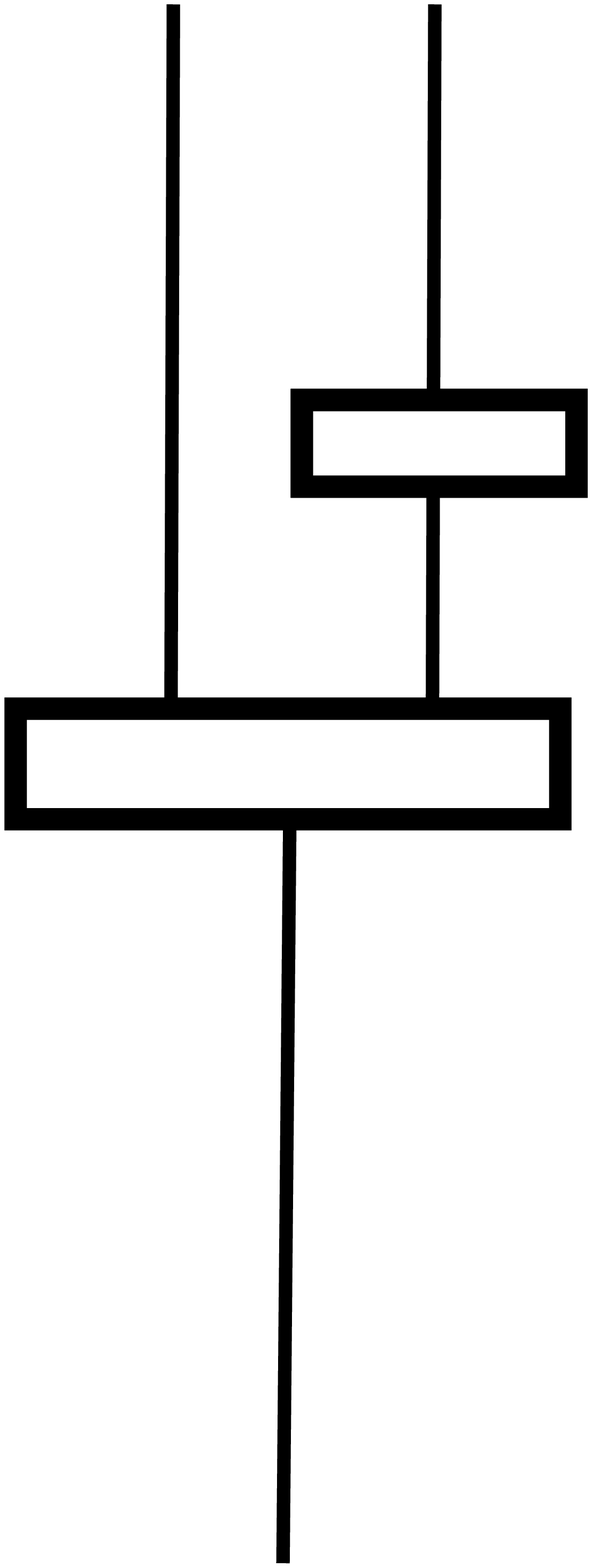}}
        \put(-34,+82){\footnotesize{$n$}}
        \put(-19,+82){\footnotesize{$m$}}
        \put(-46,20){\footnotesize{$m+n$}}
   \end{minipage}
  =
     \begin{minipage}[h]{0.09\linewidth}
        \vspace{0pt}
        \scalebox{0.115}{\includegraphics{idempotent2}}
        \put(-46,20){\footnotesize{$m+n$}}
   \end{minipage},
  \hspace{10pt} 
\Delta_{n}=
  \begin{minipage}[h]{0.1\linewidth}
        \vspace{0pt}
        \scalebox{0.12}{\includegraphics{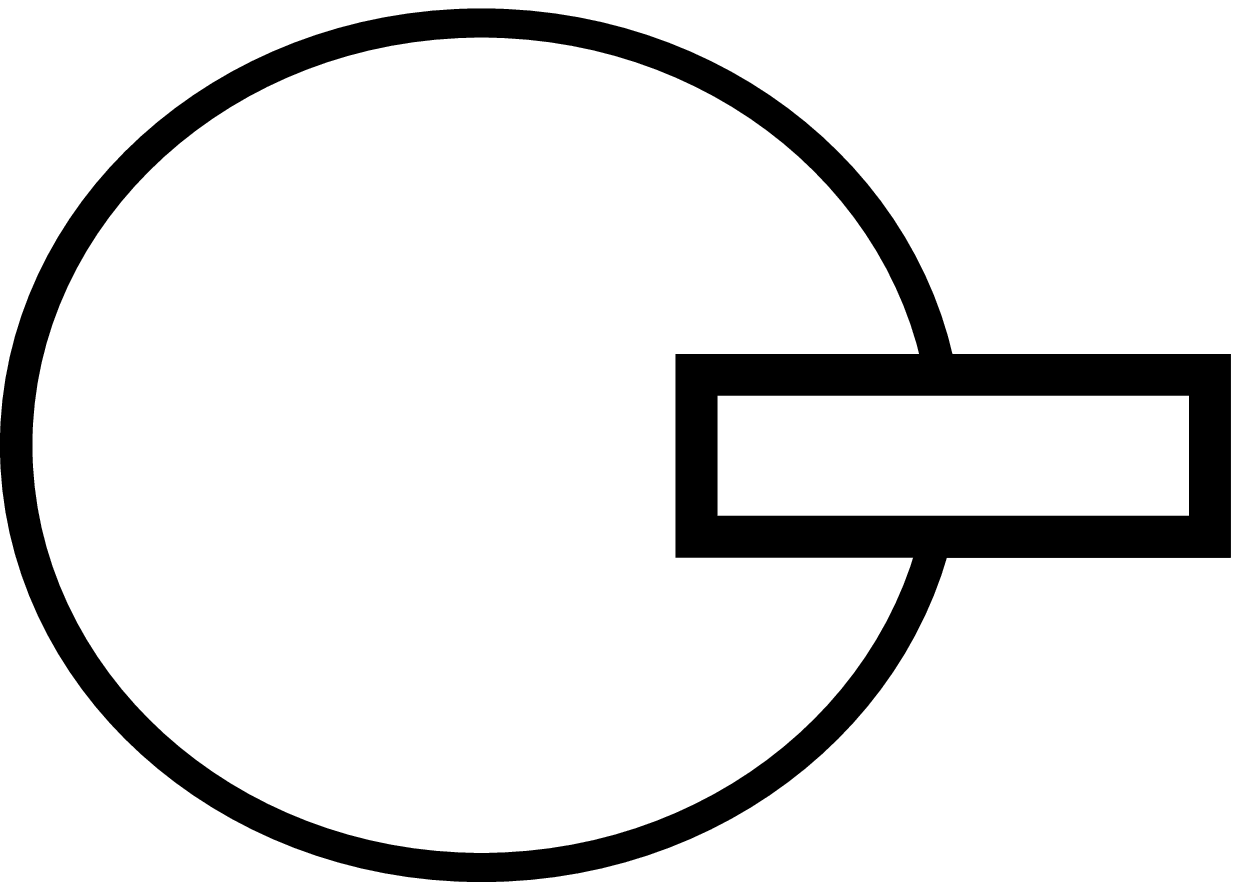}}
        \put(-29,+34){\footnotesize{$n$}}
   \end{minipage}   
  \end{eqnarray}
  \subsection{The Colored Jones Polynomial} \label{sec3}
The singular knot invariant that we are studying in this paper is a generalization of the colored Jones polynomial for classical knots. For this reason we quickly review the basics of the colored Jones polynomial.  Given a framed link $L$ in $S^3$. We decorate every component of $L$, according to its framing, by the $n^{th}$ Jones-Wenzl idempotent and take the evaluation of the decorated framed link as an element of $\mathcal{S}(S^3)$. Up to a power of $\pm A$, this depends on the framing of $L$, the value of this element is defined to be the $n^{th}$ (unreduced) colored Jones polynomial $\tilde{J}_{n,L}(A)$. One recover the reduced Jones polynomial by a change of variable and a division by $\Delta_n$. Namely,
 \begin{equation}
 \label{change of variable}
 J_{n+1,L}(q)=\frac{\tilde{J}_{n,L}(A)}{\Delta_n}\bigg|_{A=q^
 	{1/4}}
 \end{equation}
One of the primary focus of this article is the coefficients stability of an extension of the colored Jones polynomial to singular knots. We give more details about the stability properties that the colored Jones polynomial satisfies in section \ref{sec5}. 
\section{Singular Knots and the Colored Jones Polynomial}
\label{sec4}
We give a quick introduction to the basics of singular knot theory. For more details see \cite{Gemein} and \cite{Fiedler}.
A \textit{singular link} on $n$ components is the image of a smooth immersion of $n$ circles in $S^3$ that has finitely many double points. The double points are usually called \textit{singularities}. See Figure \ref{crossings}. 

\begin{figure}[H]
		\centering
	{\includegraphics[scale=0.25]{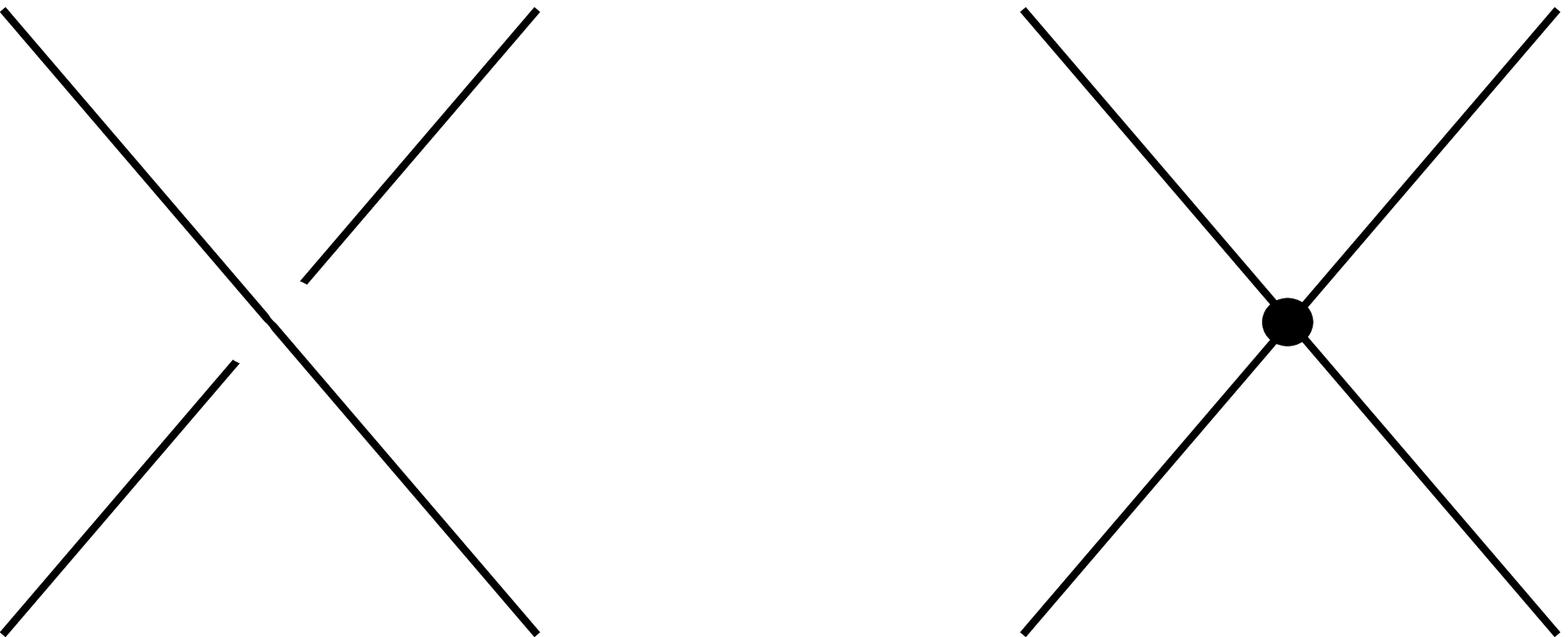}
		\caption{Regular and singular crossings}
		\label{crossings}}
\end{figure}

Singular knots are also called \textit{rigid $4$-valent graphs}. The double points are then referred to as vertices. Two singular knots are ambient isotopic if there is an orientation preserving self-homeomorphism of $S^3$ that takes one link to the other and preserves a small rigid disk around each vertex. In this paper we  work with singular link diagrams which are projections of the link on the plane such that the information at each crossing is preserved by leaving a little break in the lower strand. In this context a version of Reidemeister's theorem holds for singular links. Namely, two singular links $L_1$ and $L_2$ are ambient isotopic if and only if one can obtain a diagram of $L_1$ from a diagram of $L_2$ by a finite sequence of classical and singular Reidemeister moves shown in Figure \ref{Rmoves}.

\begin{figure}[H]
  \centering
   {\includegraphics[scale=0.1]{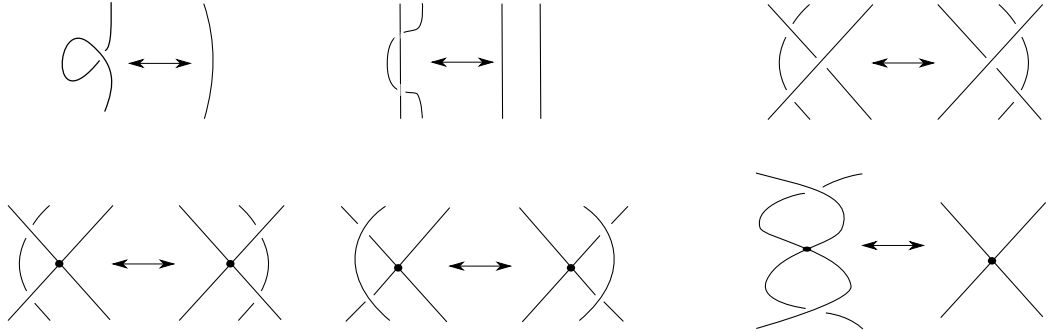}
     \caption{Classical and singular Reidemeister moves. The top three moves are $RI$, $RII$ and $RIII$ are the classical moves and the bottom moves denoted $RIV$ (two diagrams on the left) and $RV$ are the singular Reidemeister.}
  \label{Rmoves}}
\end{figure}
Similar to the case of classical knot theory, if one does not allow the Reidemeister move $RI$ then one obtains what is called  {\it regular isotopy} of singular links.

Singular knot theory have gained a lot of interest in the past two decades. This was primarily motivated by Vassiliev invariants \cite{Vassiliev}. In particular, most classical knot theory invariants have been extended to the singular versions. For instance, Fiedler \cite{Fiedler} extended the Jones and Alexander polynomials to singular knot invariants. In \cite{JL} Juyumaya and Lambropoulou constructed  Jones-type invariant for singular links using a Markov trace on a version of the Hecke algebra. Gemein \cite{Gemein} studied certain extensions of the Artin and the Burau representations to the singular braid monoid. Bataineh, Elhamdadi and Hajij extended the colored Jones polynomial definition to singular knots in \cite{bataineh2016colored}.  Churchill, Elhamdadi, Hajij and Nelson showed that the set of colorings by some algebraic structures is an invariant of unoriented singular links and used it to distinguish several singular knots and links in \cite{singularquandle}. This work was extended to oriented singular knots in \cite{singularquandle1}. Kauffman and Vogel \cite{kaufvog} extended the Dubrovnik polynomial to an invariant of singular knots in $\mathbb{R}^3$. 
Kauffman gave a one variable specialization of the Kauffman-Vogel polynomial utilizing the Jones-Wenzl projector. We will denote this invariant by $[.]_2$. This invariant is defined via the following axioms: 
\begin{enumerate}
\item

$   \left[
   \begin{minipage}[h]{0.08\linewidth}
        \vspace{0pt}
        \scalebox{0.15}{\includegraphics{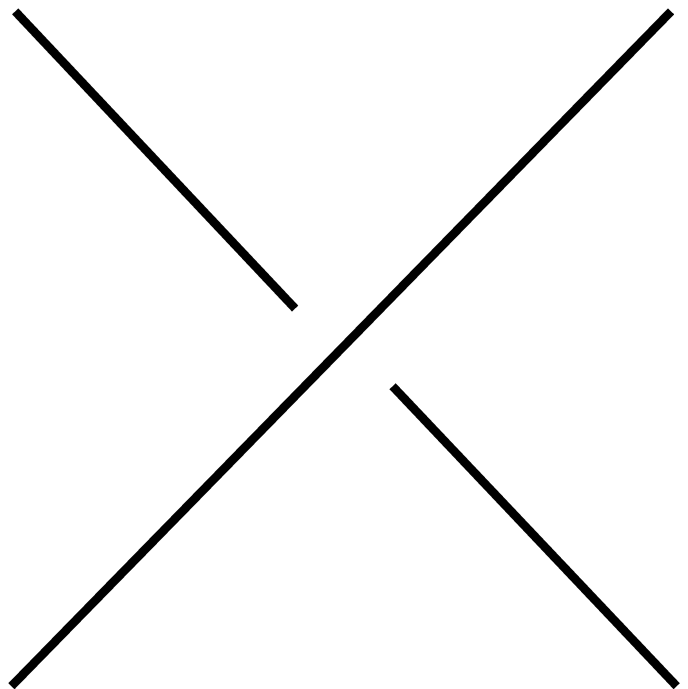}}
   \end{minipage}
  \right]_2  = \begin{minipage}[h]{0.08\linewidth}
        \vspace{0pt}
        \scalebox{0.08}{\includegraphics{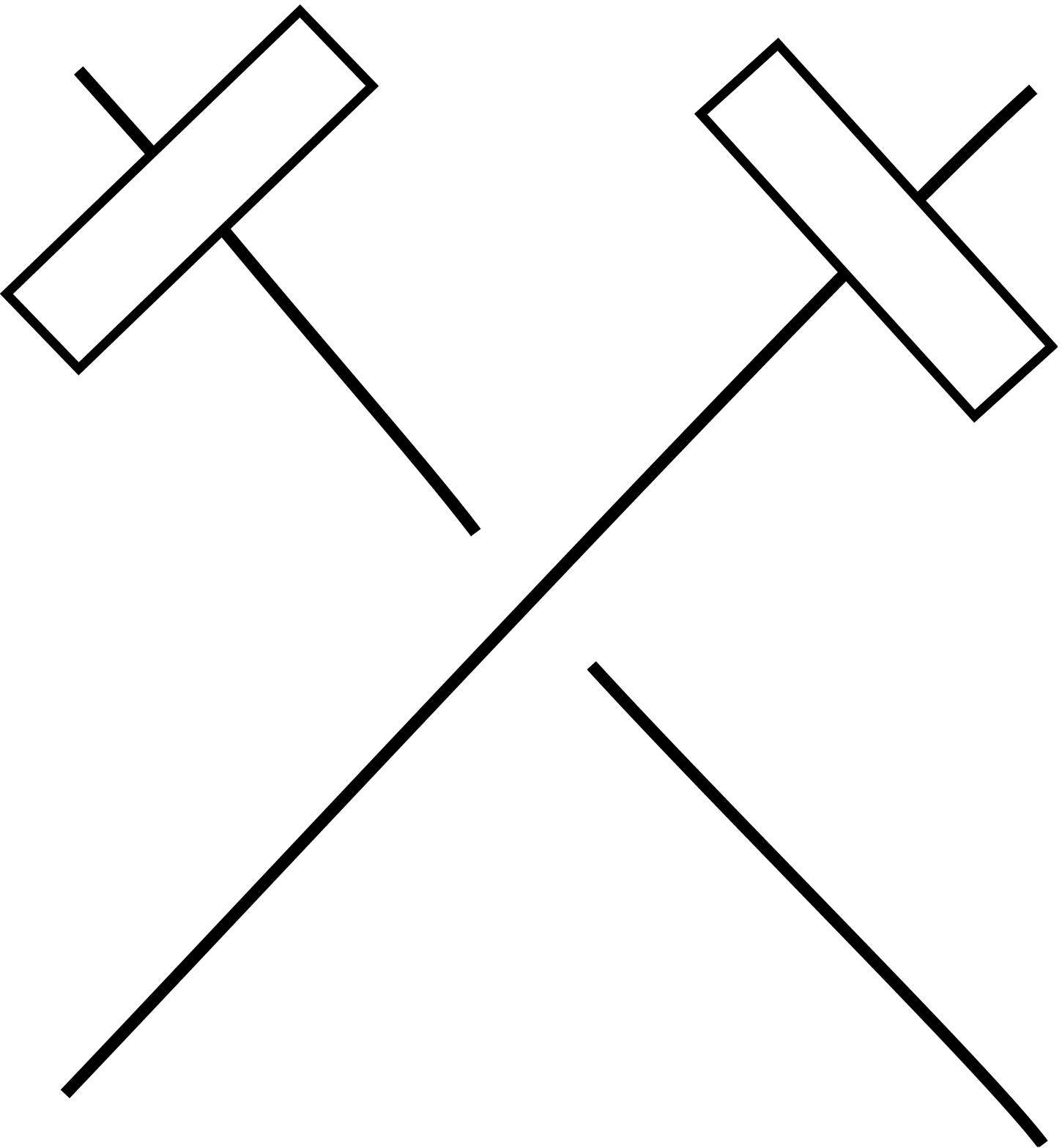}}
        \tiny{
        \put(-1,45){$2$}
        \put(-48,45){$2$}}
   \end{minipage}$
 \item
\vspace{3pt}  
 $ \left[
   \begin{minipage}[h]{0.08\linewidth}
        \vspace{0pt}
        \scalebox{0.08}{\includegraphics{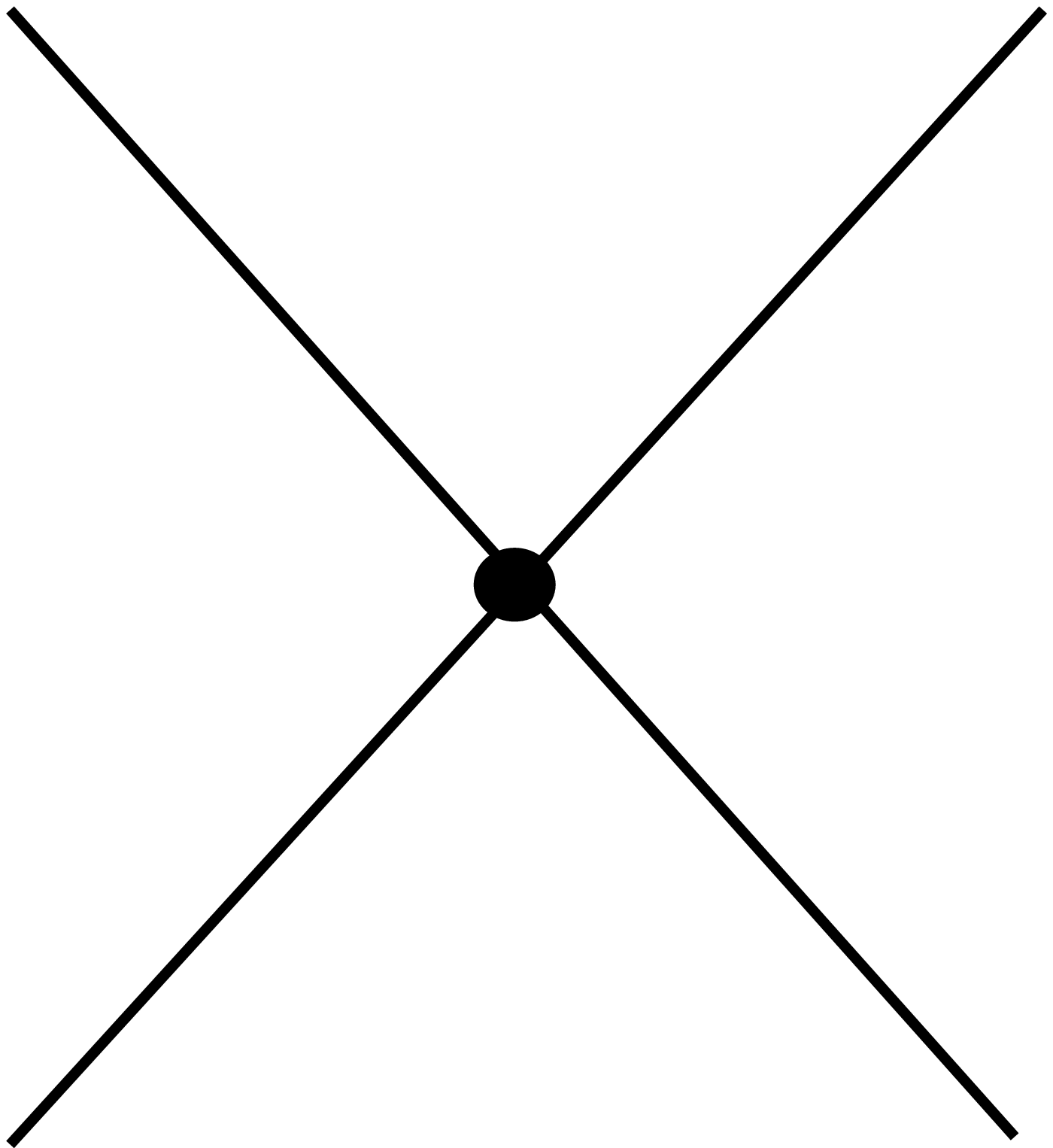}}
   \end{minipage}
  \right]_2=\begin{minipage}[h]{0.08\linewidth}
        \vspace{0pt}
        \scalebox{0.08}{\includegraphics{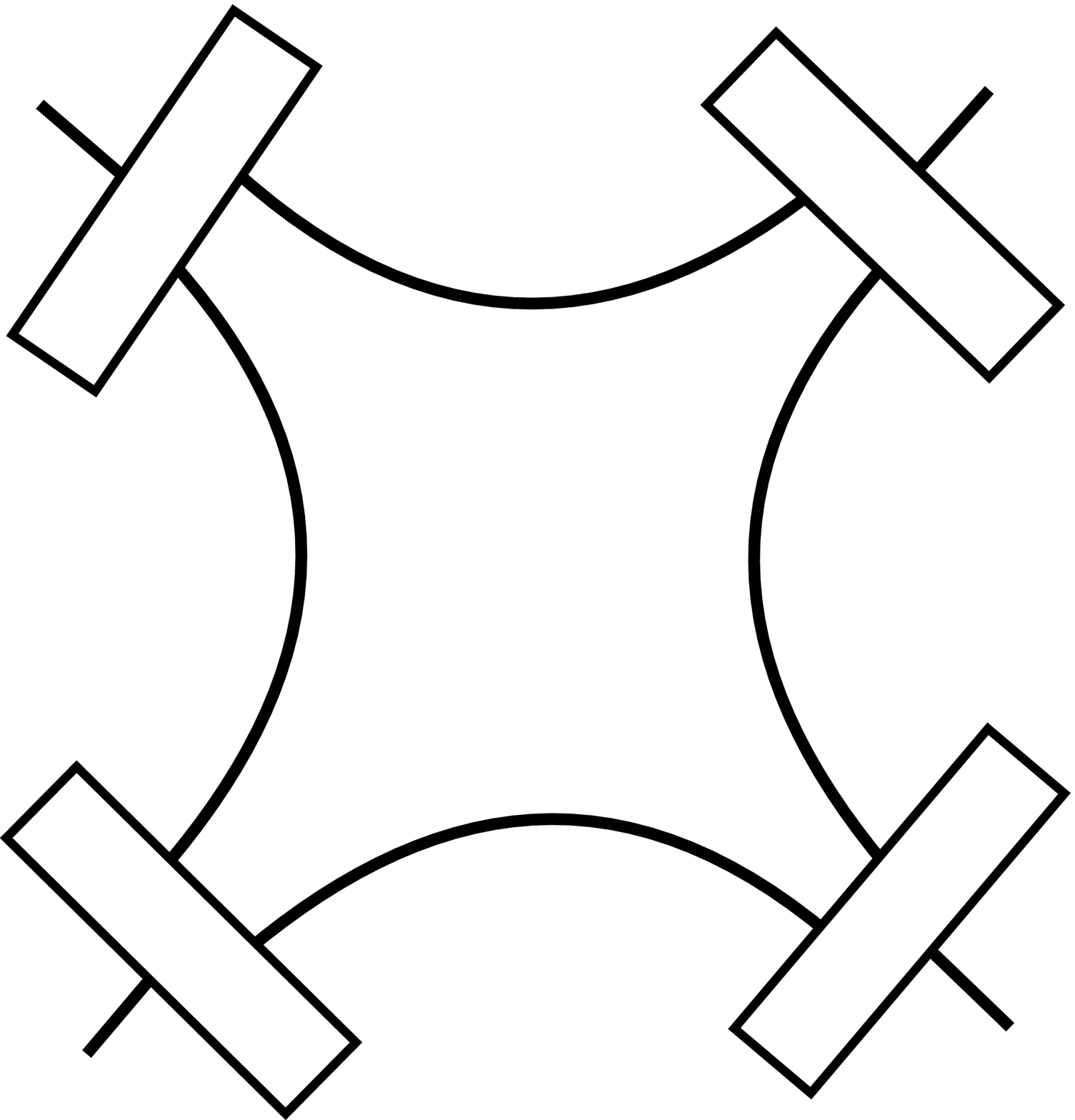}}
        \tiny{
        \put(-28,40){$1$}
        \put(-28,5){$1$}
        \put(-12,20){$1$}
        \put(-40,20){$1$}
        \put(-1,45){$2$}
        \put(-48,45){$2$}
        \put(-1,0){$2$}
        \put(-48,0){$2$}}
   \end{minipage}$ 
 \item
\vspace{3pt}  
 $ \left[
   \begin{minipage}[h]{0.08\linewidth}
        \vspace{0pt}
        \scalebox{0.08}{\includegraphics{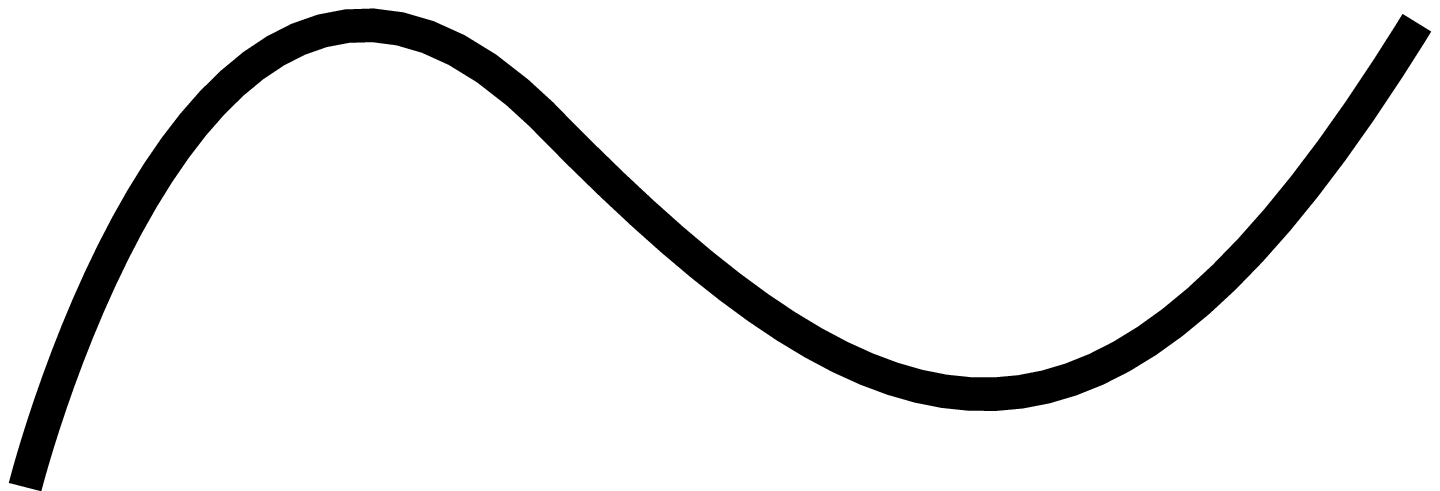}}
   \end{minipage}
  \right]_2=\begin{minipage}[h]{0.08\linewidth}
        \vspace{0pt}
        \scalebox{0.1}{\includegraphics{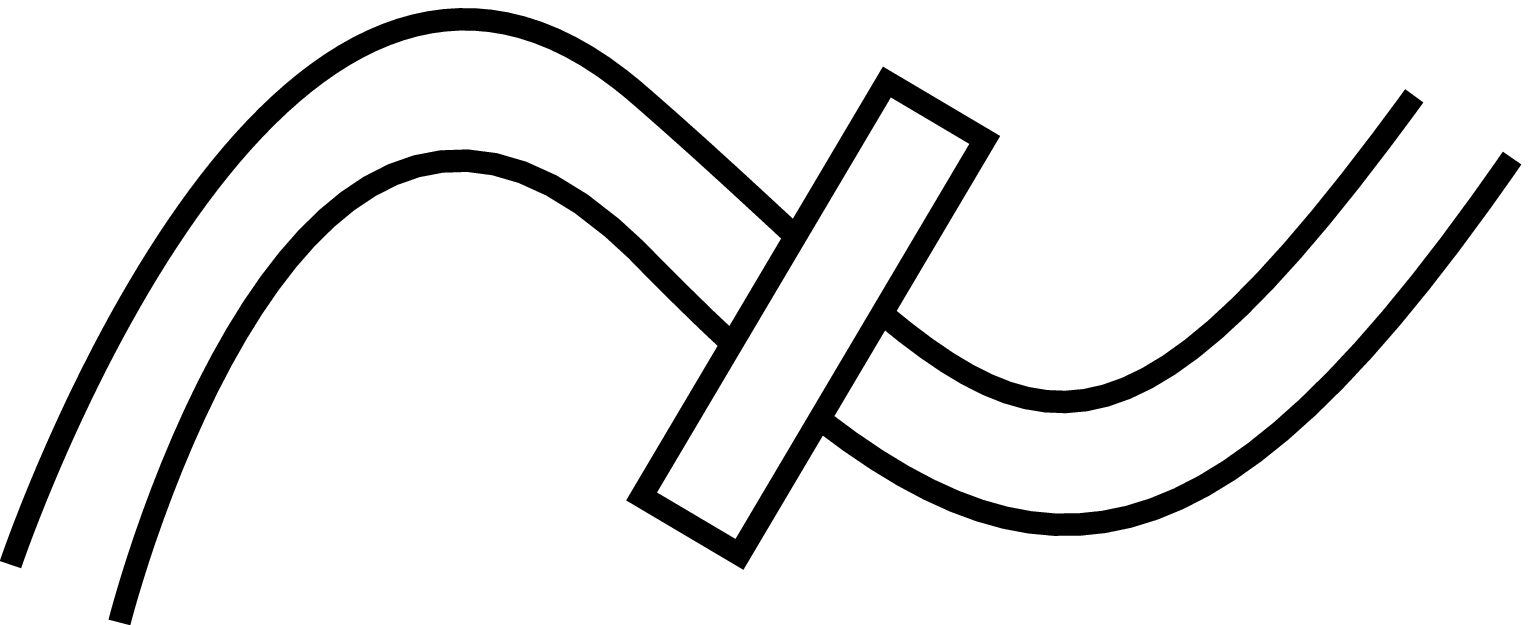}}
   \end{minipage}$     
  
\end{enumerate}	

In \cite{bataineh2016colored} we gave a natural  generalization for $[.]_2$ as follows.
\begin{definition}
\label{main definition}
Let $L$ be a singular link. For an integer $n \geq 1 $, the rational function $[L]_{2n}$ can be defined by the following rules:  

\begin{enumerate}
\item

$   \left[
   \begin{minipage}[h]{0.08\linewidth}
        \vspace{0pt}
        \scalebox{0.15}{\includegraphics{pos_crossing}}
   \end{minipage}
  \right]_{2n}  = \begin{minipage}[h]{0.08\linewidth}
        \vspace{0pt}
        \scalebox{0.08}{\includegraphics{colored_corssing_2}}
        \tiny{
        \put(-1,45){$2n$}
        \put(-48,45){$2n$}}
   \end{minipage}$
 \item
\vspace{3pt}  
 $ \left[
   \begin{minipage}[h]{0.08\linewidth}
        \vspace{0pt}
        \scalebox{0.08}{\includegraphics{singular}}
   \end{minipage}
  \right]_{2n}=\begin{minipage}[h]{0.08\linewidth}
        \vspace{0pt}
        \scalebox{0.08}{\includegraphics{singular_map}}
        \tiny{
        \put(-28,40){$n$}
        \put(-28,5){$n$}
        \put(-12,20){$n$}
        \put(-40,20){$n$}
        \put(-1,45){$2n$}
        \put(-48,45){$2n$}
        \put(-1,-3){$2n$}
        \put(-48,-3){$2n$}}
   \end{minipage}$ 
 \item
\vspace{8pt}  
 $ \left[
   \begin{minipage}[h]{0.08\linewidth}
        \vspace{0pt}
        \scalebox{0.08}{\includegraphics{strand}}
   \end{minipage}
  \right]_{2n}=\begin{minipage}[h]{0.08\linewidth}
        \vspace{0pt}
        \scalebox{0.1}{\includegraphics{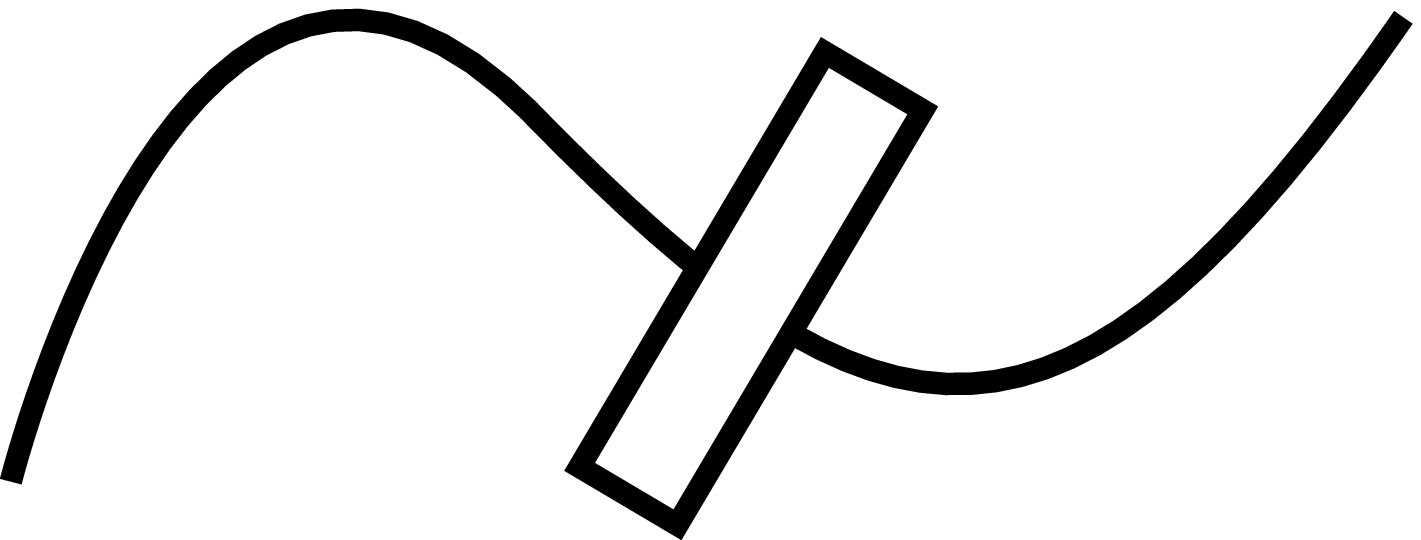}}
        \tiny{ \put(-3,5){$2n$}}
   \end{minipage}$     
  
\end{enumerate}		
\end{definition}
The proof that the previous three relations gives an invariant for singular links in $S^2$ can be done by showing that $[.]_{2n}$ is invariant under the singular Reidemeister moves. The details of this proof can be found in \cite{EH}.
It is clear that the invariant $[.]_{2n}$ can be viewed as an extension of the unreduced colored Jones polynomial for links in $S^3$. Namely, for a zero-framed knot $K$ in $S^3$ we have $\tilde{J}_{2n,K}=[K]_{2n}$. For this reason, we will denote this invariant by the most common notation of the unreduced colored Jones polynomial, that is $\tilde{J}_{2n}$. On the other hand, when computing the tail of the colored Jones polynomial of singular links we prefer to work with the normalized version of the colored Jones polynomial. We define the normalized colored Jones polynomial of a singular link $K$ by: $$J_{2n+1,K}(q)=\frac{1}{\Delta_{2n}}[K]_{2n} \bigg|_{A=q^
 	{1/4}}. $$ 
    This definition can be seen as an extension of the definition of the normalized colored Jones polynomial from the classical links. 
    
\subsection{Computing the Colored Jones Polynomial of Singular Links}

In \cite{MV} Masbum and Vogel gave an algorithm to compute the colored Jones polynomial using colored trivalent graphs. We review their algorithm and we show how it can be extended to compute the colored Jones polynomial of singular knots. We recall first some identities and definitions from  \cite{MV}.

Consider the skein module of $I \times I$ with $a+b+c$ specified points on the boundary. Partition the set
of the $a+b+c$ points on the boundary of the disk into $3$ sets of  $a$, $b$ and  $c$  
points respectively and at each cluster of points we place an appropriate idempotent, i.e. the one whose color matches the cardinality of this cluster. The skein module constructed this way will be denoted by $T_{a,b,c}.$ The skein module $T_{a,b,c}$ is either zero dimensional or one dimensional. The skein module $T_{a,b,c}$ is one dimensional if and only if the element shown in Figure \ref{taw} exists. For this element to exist it is necessary to find
non-negative integers $x,y$ and $z$ such that $a=x+y$, $b=x+z$ and $c=y+z$. 
\begin{figure}[H]
  \centering
   {\includegraphics[scale=0.2]{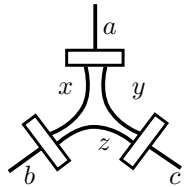}
   \small{
    \put(-60,-4){$b$}
         \put(-30,53){$a$}
          \put(-47,30){$x$}
          \put(-19,30){$y$}
          \put(-32,9){$z$}
         \put(-5,-4){$c$}
         }
     \caption{The skein element $\tau_{a,b,c}$ in the space $T_{a,b,c\text{ }}$ }
  \label{taw}}
\end{figure}
The following definition characterizes the existence of this skein element in terms of the integers $a$, $b$ and $c$. 
\begin{definition}
\label{admi}
A triple of non-negative integers $(a,b,c)$ is \textit{admissible} if $a+b+c$ is even and $a+b\geq
c\geq |a-b|.$
\end{definition}

When the triple $(a,b,c)$ is admissible, one can write  
$x=\frac{a+b-c}{2}$,
$y=\frac{a+c-b}{2}$, and 
$z=\frac{b+c-a}{2}$.  In this case we will denote the skein element that generates the space by $\tau_{a,b,c}$. We will call the triple $(a,b,c)$ the \textit{ interior colors} of $\tau_{a,b,c}$ and the triple $(x,y,z)$ the \textit{interior colors} of $\tau_{a,b,c}$. Note that when the triple $(a,b,c)$ is not admissible then the space $T_{a,b,c}$ is zero dimensional.

The fact that the inside colors are determined by the outside colors allows us to replace $\tau _{a,b,c}$ by a trivalent graph as follows:
\begin{figure}[H]

  \centering
   {\includegraphics[scale=0.29]{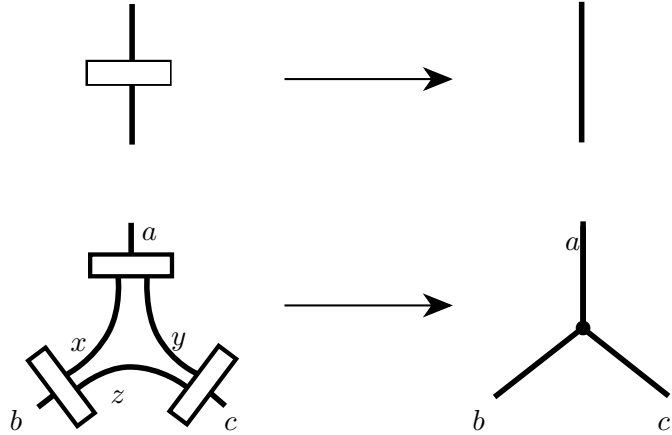}
    \put(-75,0){$b$}
         \put(-40,68){$a$}
          \put(-227,30){$x$}
          \put(-189,32){$y$}
          \put(-212,10){$z$}
         \put(-5,0){$c$}
            \put(-250,0){$b$}
         \put(-200,72){$a$}
               \put(-169,0){$c$}
     \caption{The skein element and its corresponding  trivalent vertex}
  \label{taw1}}
\end{figure}
 A \textit{colored trivalent graph} in $S^3$ is an embedded trivalent graph in $S^3$ with edges labeled by non-negative integers. One usually uses the word \textit{color} to refer to a label of the edge of a trivalent graph. A colored trivalent graph is called \textit{admissible} if the three edges meeting at a vertex satisfy the admissibility condition of the definition \ref{admi}. If $D$ is an admissible colored trivalent graph then the Kauffman bracket evaluation of $D$ is defined to be the evaluation of $D$ as an element in $\mathcal{S}(S^{2})$ after replacing each edge colored $n$ by the projector $f^{(n)}$ and each admissible vertex colored $(a,b,c)$ by the skein element $\tau_{a,b,c}$, as in Figure \ref{taw1}. If a colored trivalent graph has a non-admissible vertex then we will consider its evaluation in $\mathcal{S}(S^{2})$ to be zero.

 We will need the evaluation of the following important colored trivalent graphs shown in Figure \ref{graphs}.

   \begin{figure}[H]
  \centering
   {\includegraphics[scale=0.13]{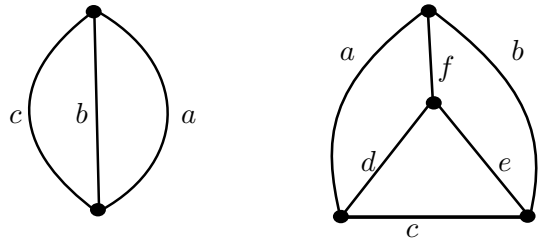}
    \put(-75,+62){$a$}
          \put(-10,+62){$b$}
          \put(-50,-5){$c$}
          \put(-67,20){$d$}
          \put(-15,20){$e$}
          \put(-38,56){$f$}
          \put(-135,38){$a$}
          \put(-175,38){$b$}
          \put(-200,38){$c$}
            \caption{The theta graph on the left and the tetrahedron graph on the right.}\label{graphs}}
\end{figure}
For an admissible triple $(a,b,c)$, an explicit formula for the \textit{theta coefficient}, denoted $\theta(a,b,c)$, was computed in \cite{MV} and is given by:
\begin{equation}
  \begin{minipage}[h]{0.1\linewidth}
       \scalebox{0.10}{\includegraphics{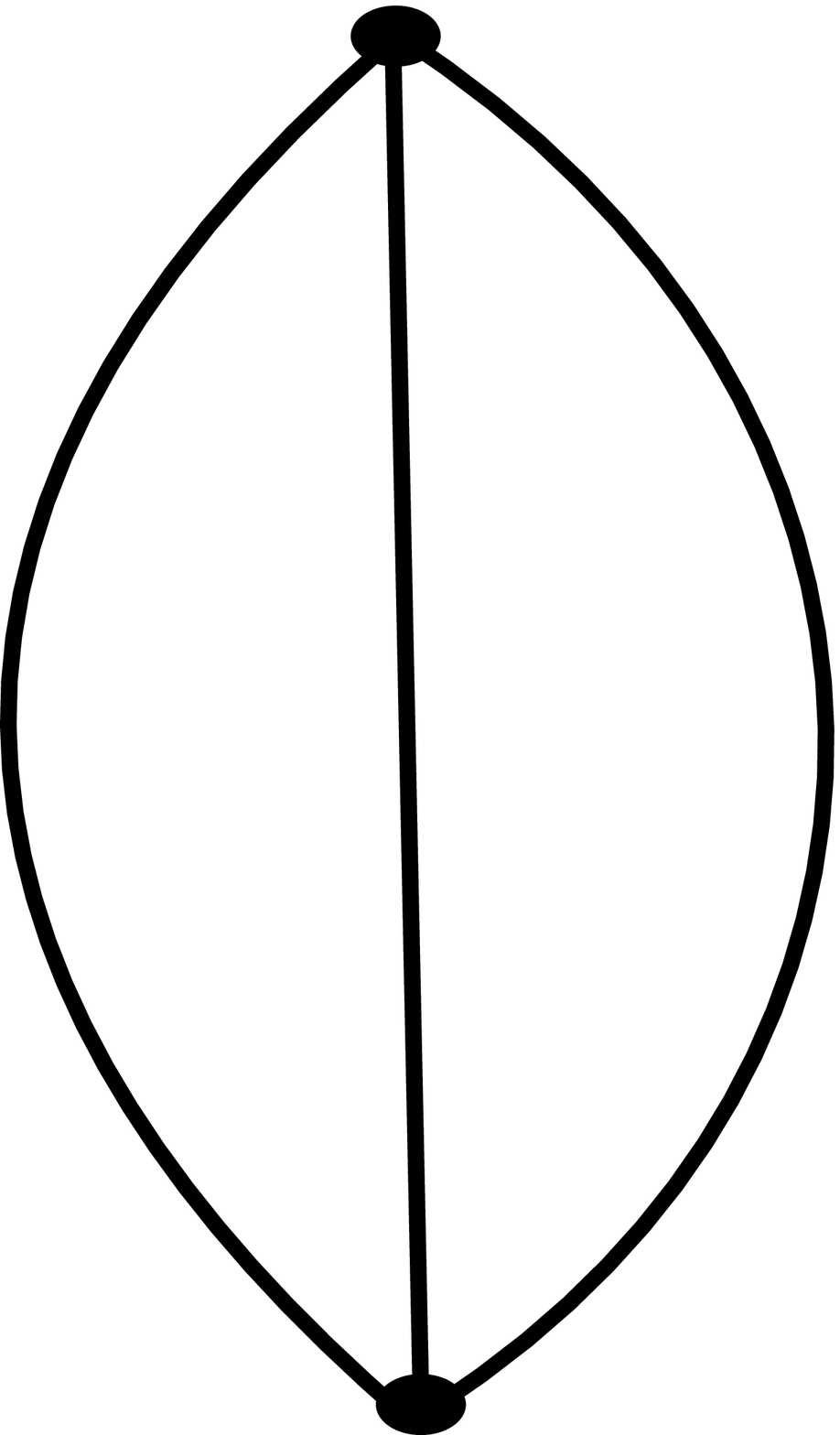}}
       \put(-42,25){$a$}
       \put(-24,25){$b$}
       \put(-10,25){$c$}
   \end{minipage}
   =(-1)^{x+y+z}\frac{[x+y+z+1]![x]![z]![y]!}{[x+y]![x+z]![y+z]!}
   \end{equation}
where $x,y$ and $z$ are the interior colors of the vertex $(a,b,c)$. In terms of the Pochhammer symbol the previous identity is given by 

\begin{equation}
\label{pac}
\theta(a,b,c)=(-1)^{x+y+z} q^{-(x+y+z)/2}  
 \frac{(q;q)_{x}(q;q)_{y}(q;q)_{z} (q ;q)_{x+y+z+1}}{ (1-q)(q;q)_{x+y}(q;q)_{y+z}(q;q)_{x+z}},
\end{equation}
where

\begin{equation}
(q;q)_n=\prod_{i=0}^{n-1}(1-q^{i+1}).
\end{equation}
The \textit{tetrahedron coefficient} is defined to be the evaluation of the graph appearing on the right handside of Figure 
\ref{graphs} 
and a formula of it can be found in \cite{MV}. The tetrahedron graph in Figure \ref{graphs} is denoted by $Tet\left[ 
\begin{array}{ccc}
a & d & e \\ 
f & c & b%
\end{array}%
\right]$. Besides the previous two coefficients the following two identities hold in $T_{a,b,c}$:

\begin{eqnarray}
\label{firsty}
    \begin{minipage}[h]{0.21\linewidth}
         \vspace{-0pt}
         \scalebox{0.35}{\includegraphics{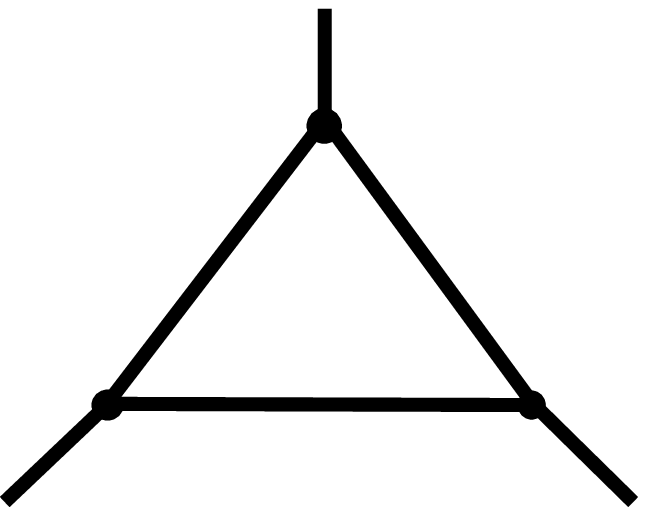}}
     \put(-69,-10){$b$}
          \put(-1,-10){$c$}
         \put(-26,50){$a$}
         \put(-59,25){$d$}
         \put(-16,25){$e$}
          \put(-36,0){$f$}
                    \end{minipage}&=&\frac{Tet\left[ 
\begin{array}{ccc}
a & d & e \\ 
f & c & b%
\end{array}%
\right]}{\theta(a,b,c)}
   \begin{minipage}[h]{0.16\linewidth}
        \vspace{-0pt}
        \scalebox{0.25}{\includegraphics{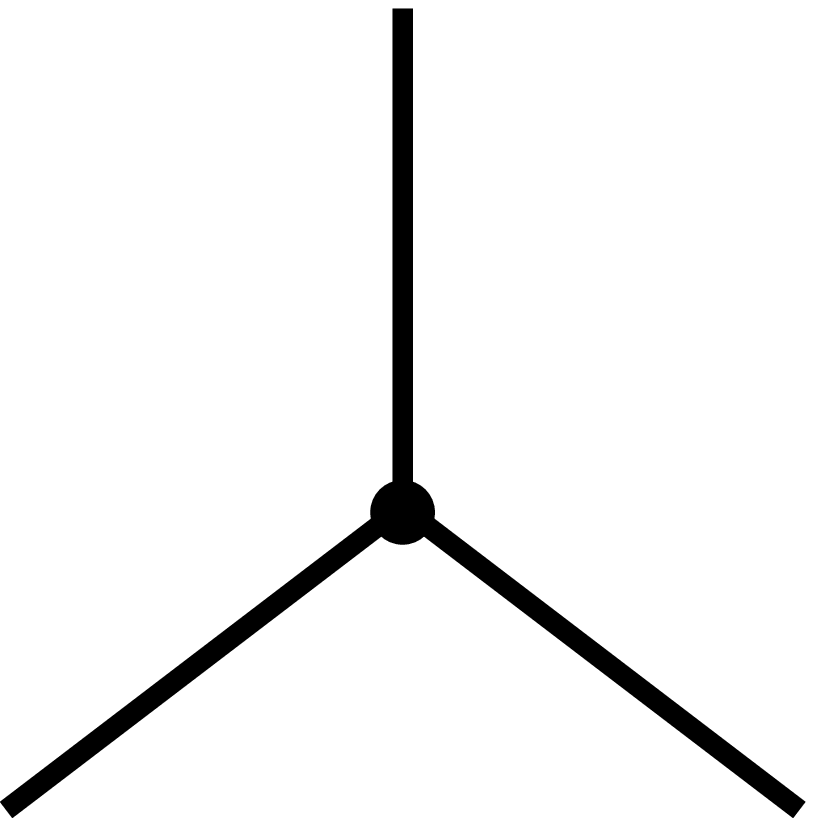}}
         \put(-60,-10){$b$}
          \put(-1,-10){$c$}
         \put(-26,50){$a$}
           \end{minipage}
  \end{eqnarray}

and

 \begin{eqnarray} 
 \label{trivalent identity}
   \begin{minipage}[h]{0.1\linewidth}
        \vspace{0pt}
        \scalebox{0.15}{\includegraphics{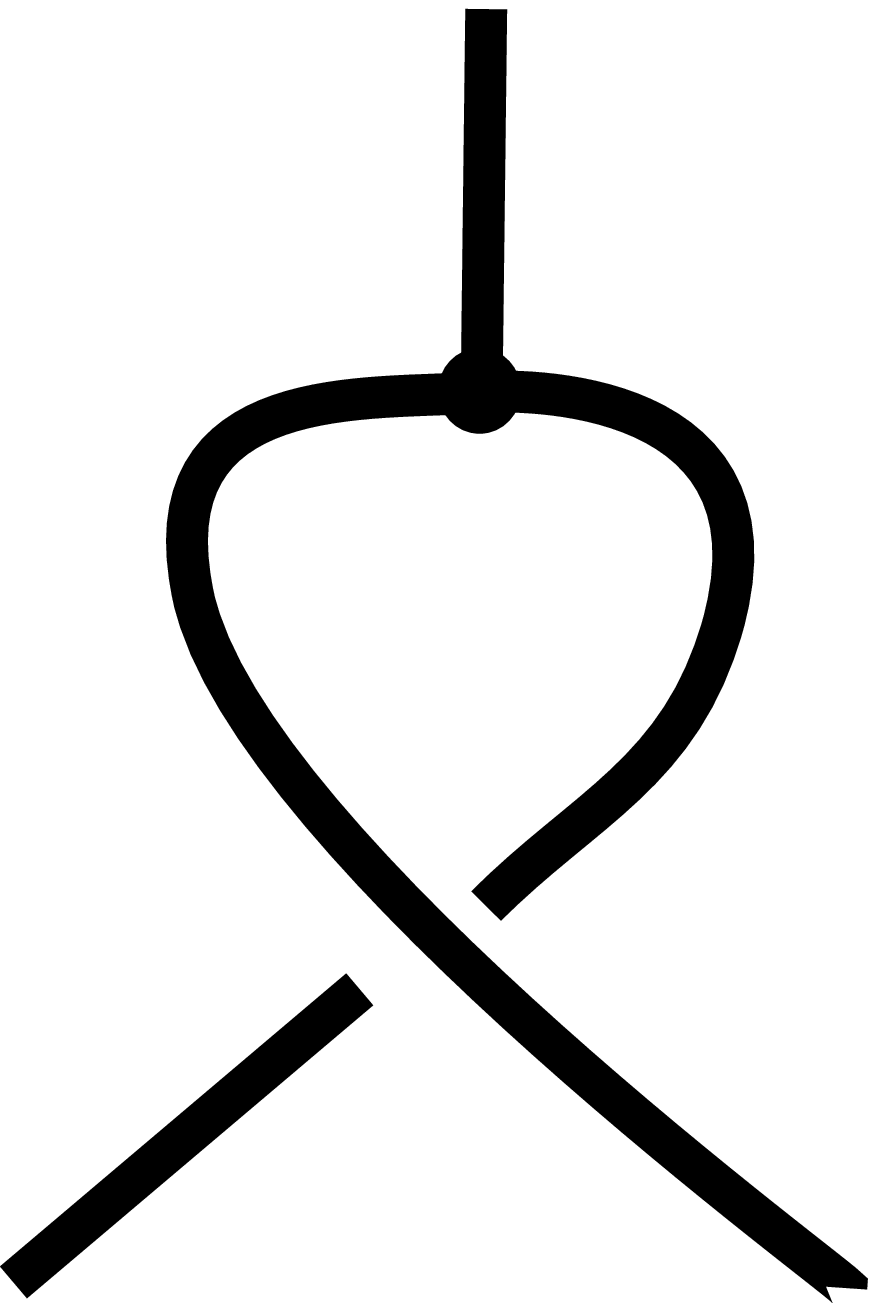}}
         \put(-45,-6){$b$}
          \put(-1,-6){$c$}
         \put(-26,50){$a$}

   \end{minipage} &=& \lambda^a_{b,c} \quad
   \begin{minipage}[h]{0.16\linewidth}
        \vspace{-0pt}
        \scalebox{0.25}{\includegraphics{3-valent_graph}}
         \put(-60,-10){$b$}
          \put(-1,-10){$c$}
         \put(-26,50){$a$}
           \end{minipage},
   \end{eqnarray}
   
where $\lambda^a_{b,c}=(-1)^{(a+b-c)/2}A^{(a^\prime+b^\prime-c^\prime)/2}$, and $x^{\prime}=x(x+2)$.

Define the space $T_{a,b}$ similar to the skein module $T_{a,b,c}$. Namely, this module is the submodule of the skein module of the disk with $a+b$ marked point on the boundary and place the idempotents $f^{a}$ and $f^{b}$ on the appropriate sets of points as we did for $T_{a,b,c}$. This module is also zero dimensional or one dimensional. Using the properties of the idempotent one can see that this space is one dimensional if and only if $a=b$ and zero dimensional otherwise.  In $T_{a,b}$ the following identity also holds:

\begin{eqnarray}
\label{bubble_1}
    \begin{minipage}[h]{0.09\linewidth}
         \vspace{-0pt}
         \scalebox{0.27}{\includegraphics{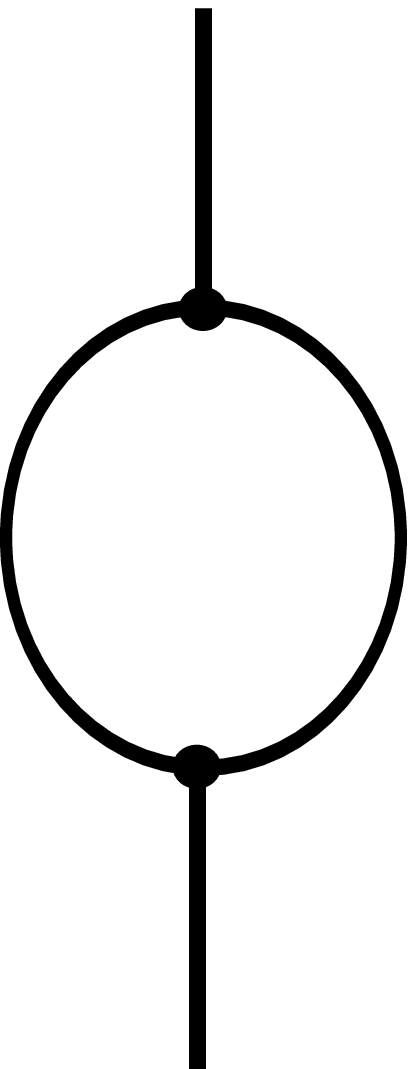}}
         \put(-10,0){$d$}
         \put(-10,80){$a$}
         \put(-40,45){$b$}
         \put(2,45){$c$}
              \end{minipage}&=&\delta_{a}^{d} \frac{\theta(a,b,c)}{\Delta_a}\quad 
   \begin{minipage}[h]{0.16\linewidth}
        \vspace{-0pt}
        \scalebox{0.1}{\includegraphics{lonelystrand}}
          \put(3,0){$a$}
           \end{minipage}
  \end{eqnarray}

We define the module of the disk $\mathscr{D}^{a,b}_{c,d}$ similar to the modules $T_{a,b,c}$ and $T_{a,b}$. See Figure \ref{disk} for an illustration.
   
   \begin{figure}[H]
  \centering
   {\includegraphics[scale=0.13]{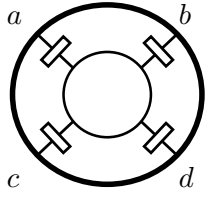}
    \put(-75,+62){$a$}
          \put(-10,+62){$b$}
          \put(-75,0){$c$}
          \put(-10,0){$d$}
            \caption{The relative skein module $\mathscr{D}^{a,b}_{c,d}$}
  \label{disk}}
\end{figure}
Now we define the bilinear form $<,>: \mathscr{D}^{a,b}_{c,d} \times \mathscr{D}^{a,b}_{c,d} \longrightarrow \mathcal{S}(S^2)$ as follows. Let $E$ and $F$ be two diagrams in $\mathscr{D}^{a,b}_{c,d}$. The diagram $<E,F>$ is an element in $\mathcal{S}(S^2)$ defined in Figure \ref{bilinear}.
   \begin{figure}[H]
  \centering
   {\includegraphics[scale=0.14]{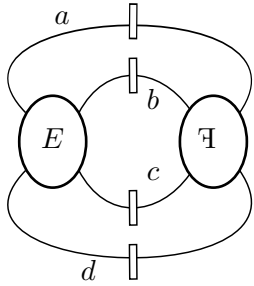}
    \put(-75,+97){$a$}
          \put(-40,+65){$b$}
          \put(-65,0){$d$}
          \put(-40,38){$c$}
          \put(-21,50){$\reflectbox{F}$}
          \put(-80,50){$E$}
            \caption{The diagram $<E,F>$ in $\mathcal{S} (S^2)$. }
  \label{bilinear}}
\end{figure}
Let $B_H=\{\mathcal{T}_i|(a,c,i),(b,d,i) \in ADM\}$ be the set of tangles defined in Figure \ref{basees1} (a). It is known that this set forms an orthogonal basis for the space $\mathscr{D}^{a,b}_{c,d}$ with respect to the bilinear form $<,>$. For more detail see \cite{Lic92}.
\begin{figure}[H]
\centering
\begin{minipage}{.5\textwidth}
  \centering
  \includegraphics[width=.25\linewidth]{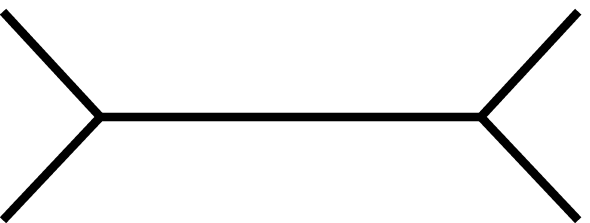}
  \put(-29,+17){\footnotesize{$i$}}
        \put(-69,+26){\footnotesize{$a$}}
        \put(-69,-1){\footnotesize{$c$}}
        \put(2,-1){\footnotesize{$d$}}
        \put(2,+26){\footnotesize{$b$}}
         \put(-32,-32){(a)}
\end{minipage}%
\begin{minipage}{.5\textwidth}
  \centering
  \includegraphics[width=.1\linewidth]{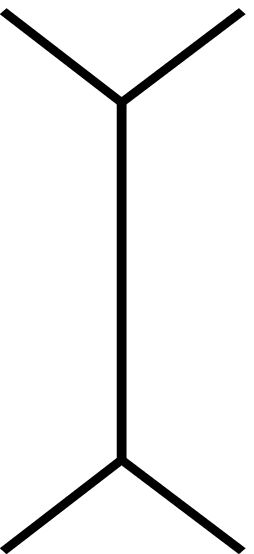}
  \put(-19,+23){\footnotesize{$i$}}
        \put(-29,+46){\footnotesize{$a$}}
        \put(-29,-1){\footnotesize{$c$}}
        \put(2,-1){\footnotesize{$d$}}
        \put(2,+46){\footnotesize{$b$}}
   \put(-15,-15){(b)}
\end{minipage}
  \caption{(a)The element $\mathcal{T}_i$ in the $B_H$ (b)The element $\mathcal{T}^{\prime}_i$ in the $B_V$ }
  \label{basees1}
\end{figure}
 By symmetry, the set $B_V=\{\mathcal{T}^{\prime}_i|(a,b,i),(c,d,i) \in ADM\}$ is also a basis. The change of basis between these two bases $B_H$ and $B_V$ is given by:

 \begin{eqnarray}
 \label{change of basis}
    \begin{minipage}[h]{0.09\linewidth}
        \vspace{0pt}
        \scalebox{0.3}{\includegraphics{VB}}
         \put(-29,+46){\footnotesize{$a$}}
        \put(-29,-1){\footnotesize{$c$}}
        \put(2,-1){\footnotesize{$d$}}
        \put(2,+46){\footnotesize{$b$}}
        \put(-5,+23){\footnotesize{$j$}}
   \end{minipage}
   &=&\displaystyle\sum\limits_{i}\left\{ 
\begin{array}{ccc}
a & b & i \\ 
c & d & j%
\end{array}%
\right\}\hspace{1 mm}  
    \begin{minipage}[h]{0.09\linewidth}
        \vspace{0pt}
        \scalebox{0.36}{\includegraphics{horizantal_basis}}
             \put(-29,+17){\footnotesize{$i$}}
        \put(-60,+26){\footnotesize{$a$}}
        \put(-60,-6){\footnotesize{$c$}}
        \put(1,-6){\footnotesize{$d$}}
        \put(1,+26){\footnotesize{$b$}}
   \end{minipage}
   \end{eqnarray}

The previous identity is also called the \textit{recoupling identity}. The coefficient $\left\{ 
\begin{array}{ccc}
a & b & i \\ 
c & d & j%
\end{array}%
\right\}$ is usually called the $6j-$ symbol. The \textit{fusion identity} is given by:

 \begin{eqnarray}
 \label{fusion}
    \begin{minipage}[h]{0.09\linewidth}
        \vspace{0pt}
        \scalebox{0.4}{\includegraphics{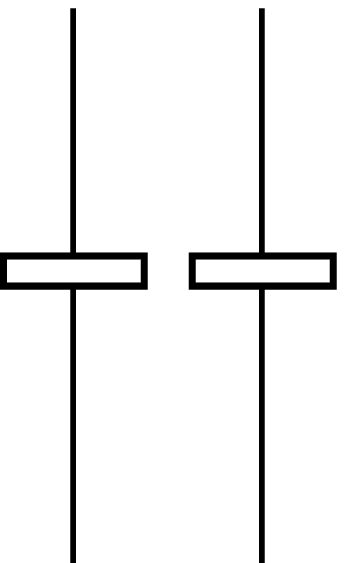}}
          \put(-5,+50){\footnotesize{$b$}}
        \put(-40,+50){\footnotesize{$a$}}
   \end{minipage}
   =\displaystyle\sum\limits_{i}\frac{\Delta_{a+b}}{\theta(a,b,i)}\hspace{1 mm}  
   \begin{minipage}[h]{0.09\linewidth}
        \vspace{0pt}
        \scalebox{0.3}{\includegraphics{VB}}
         \put(-29,+46){\footnotesize{$a$}}
        \put(-29,-1){\footnotesize{$a$}}
        \put(2,-1){\footnotesize{$b$}}
        \put(2,+46){\footnotesize{$b$}}
        \put(-5,+23){\footnotesize{$i$}}
   \end{minipage}
   \end{eqnarray}
 The fusion identity (\ref{fusion}) and identity (\ref{trivalent identity}) can be used to obtain the \textit{crossing fusion identity}:  

\begin{eqnarray}
\label{CFI}
\begin{minipage}[h]{0.08\linewidth}
        \vspace{0pt}
        \scalebox{0.08}{\includegraphics{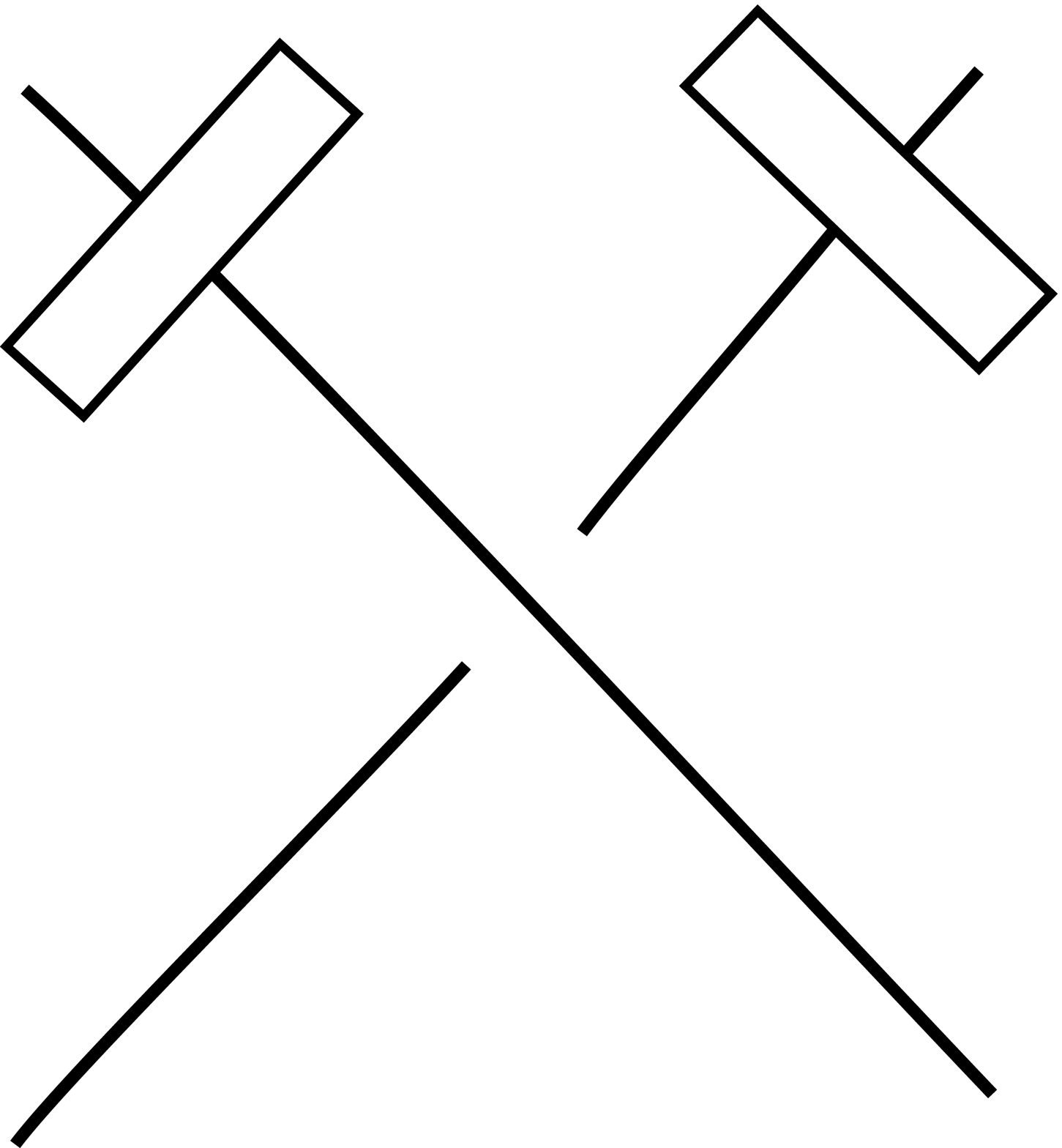}}
        \small{
        \put(-1,45){$b$}
        \put(-48,45){$a$}}
   \end{minipage} \quad =\displaystyle\sum\limits_{i} \frac{\Delta_{a+b}}{\theta(a,b,i)}\lambda^i_{a,b} \hspace{1 mm}  
   \begin{minipage}[h]{0.09\linewidth}
        \vspace{0pt}
        \scalebox{0.3}{\includegraphics{VB}}
         \put(-29,+46){\footnotesize{$a$}}
        \put(-29,-1){\footnotesize{$a$}}
        \put(2,-1){\footnotesize{$b$}}
        \put(2,+46){\footnotesize{$b$}}
        \put(-5,+23){\footnotesize{$i$}}
   \end{minipage}
 \end{eqnarray}

Let $L$ be a singular link in $S^2$. We want to compute the value of  $[L]_{2n}=\tilde{J}_{2n,L}$. Now $\tilde{J}_{2n,L}$ is a skein element in the skein module $\mathcal{S}(S^2)$ obtained by replacing every crossing by the right hand-side of rule (1) in Definition \ref{main definition} and every singular crossing by the skein element on the right hand side of rule (2) in Definition \ref{main definition}. In order to evaluate the skein element  $\tilde{J}_{2n,L}$ we show how it can be realized as a linear combination of colored trivalent graphs in $\mathcal{S}(S^2)$. The evaluation of any colored trivalent graph in $\mathcal{S}(S^2)$ can then by calculated using the algorithm given in \cite{MV}. This gives us a method of computing the evaluation of $\tilde{J}_{2n,L}$ for any singular link $L$.

The skein element $\tilde{J}_{2n,L}$ can be realized as a $\mathbb{Q}(A)$-linear combination of colored trivalent graphs in $\mathcal{S}(S^2)$ as follows:
\begin{enumerate}
\item Use the crossing fusion identity to change every crossing to a $\mathbb{Q}(A)$-linear combination of trivalent graphs as in equation (\ref{CFI}). 
\item Now we know that the singular crossings are replaced by the skein element $(2)$ in Definition \ref{main definition}. We notice that this skein element can be realized as a trivalent graph as follows:
 \begin{eqnarray*}
   \begin{minipage}[h]{0.14\linewidth}
        \vspace{0pt}
        \scalebox{0.22}{\includegraphics{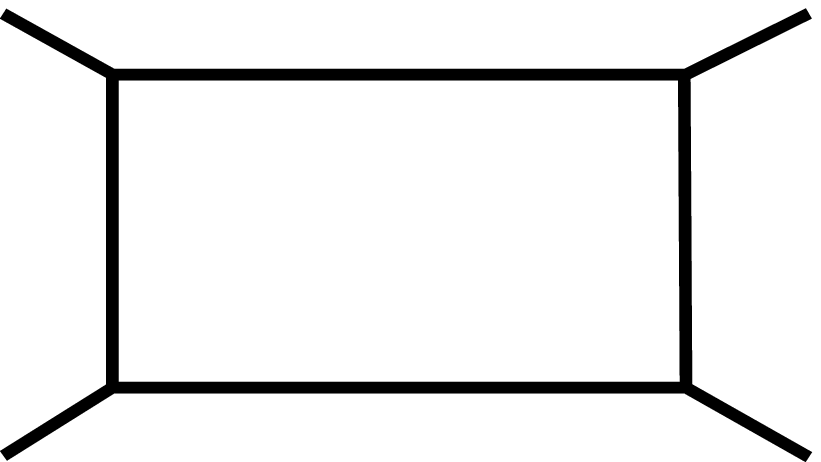}}
             \put(-29,+28){\footnotesize{$n$}}
             \put(-29,-5){\footnotesize{$n$}}
        \put(-60,+26){\footnotesize{$2n$}}
        \put(-60,-6){\footnotesize{$2n$}}
        \put(1,-6){\footnotesize{$2n$}}
        \put(1,+26){\footnotesize{$2n$}}
   \end{minipage}= 
\hspace{7 mm}  
    \begin{minipage}[h]{0.08\linewidth}
        \vspace{0pt}
        \scalebox{0.08}{\includegraphics{singular_map}}
        \tiny{
        \put(-28,40){$n$}
        \put(-28,5){$n$}
        \put(-12,20){$n$}
        \put(-40,20){$n$}
        \put(-1,45){$2n$}
        \put(-48,45){$2n$}
        \put(-1,-3){$2n$}
        \put(-48,-3){$2n$}}
   \end{minipage}
   \end{eqnarray*}
\end{enumerate}
 The evaluation of any trivalent graph in $\mathcal{S}^2$ can be calculated by using an algorithm that utilizes the recoupling formula and identities (\ref{firsty}) and (\ref{bubble_1}). The details of this algorithm can be found in \cite{MV}. We give an example to illustrate how the invariant $\tilde{J}_{{2n},L}$ can be computed for a singular link $L$ using this method.

\begin{example}
To illustrate how the invariant $\tilde{J}_{2n,L}$ can be computed in practice we compute the example given in Figure \ref{Singular Torus knot}. We denote this singular knot by $ST(k,l)$. When $l=0$ we will denote this knot simply by $ST_k$.
\begin{figure}[H]
   {\includegraphics[scale=0.08]{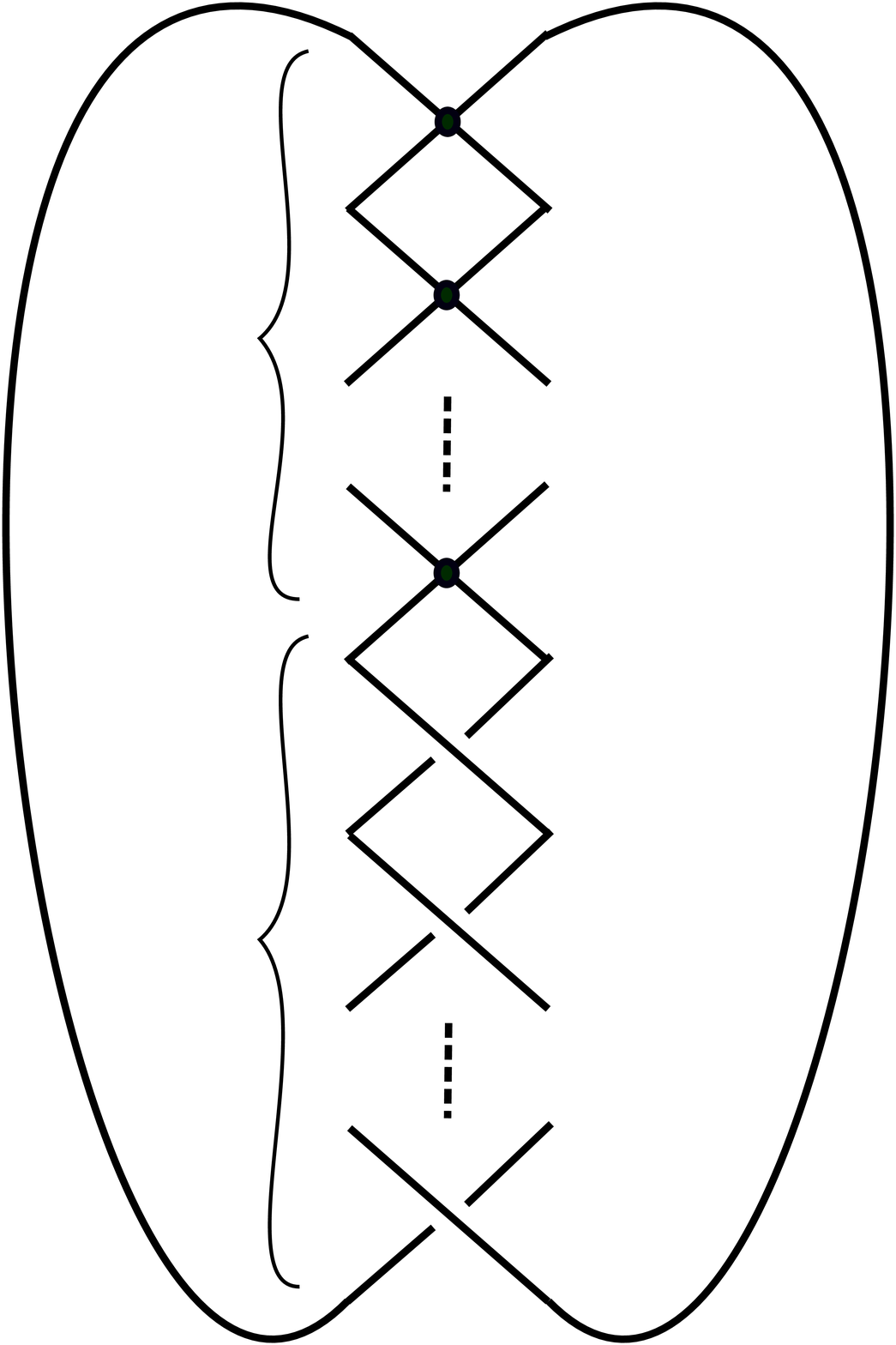}
   \put(-70,80){$k$}
   \put(-70,40){$l$}
     \caption{Singular Torus $ST(k,l)$.}
  \label{Singular Torus knot}} 
\end{figure}

To compute the value of this invariant we first notice that for positive integers $k$ and $n$ the following skien identity holds:

 \begin{eqnarray}  
  \left( \hspace{10pt}
   \begin{minipage}[h]{0.11\linewidth}
        \vspace{0pt}
        \scalebox{0.22}{\includegraphics{bubble_as_spin}}
             \put(-29,+30){\footnotesize{$n$}}
             \put(-29,-5){\footnotesize{$n$}}
        \put(-60,+29){\footnotesize{$2n$}}
        \put(-60,-6){\footnotesize{$2n$}}
        \put(1,-6){\footnotesize{$2n$}}
        \put(1,+26){\footnotesize{$2n$}}
   \end{minipage} \hspace{10pt} \right)^{\otimes k}  &=& \sum\limits_{i=0}^{n} R_{n,i}
\hspace{1 mm}  
    \begin{minipage}[h]{0.09\linewidth}
        \vspace{0pt}
        \scalebox{0.36}{\includegraphics{horizantal_basis}}
             \put(-29,+17){\footnotesize{$2i$}}
        \put(-60,+26){\footnotesize{$2n$}}
        \put(-60,-6){\footnotesize{$2n$}}
        \put(1,-6){\footnotesize{$2n$}}
        \put(1,+26){\footnotesize{$2n$}}
   \end{minipage}
   \label{small claim}
   \end{eqnarray}
where

\begin{equation}
R_{n,i}=\frac{\theta (2n,2n,2i)^{k-1}  }{\theta (n,n,2i)^{k}}\Delta_{2i}
\end{equation}

To prove identity (\ref{small claim}), we apply the fusion identity to obtain:
 \begin{eqnarray*}
   \begin{minipage}[h]{0.14\linewidth}
        \vspace{0pt}
        \scalebox{0.22}{\includegraphics{bubble_as_spin}}
             \put(-29,+28){\footnotesize{$n$}}
             \put(-29,-5){\footnotesize{$n$}}
        \put(-60,+26){\footnotesize{$2n$}}
        \put(-60,-6){\footnotesize{$2n$}}
        \put(1,-6){\footnotesize{$2n$}}
        \put(1,+26){\footnotesize{$2n$}}
   \end{minipage}   = \sum\limits_{i=0}^{n} \frac{\Delta{2i}}{\theta(n,n,2i)} 
\hspace{5 mm}  
    \begin{minipage}[h]{0.1\linewidth}
        \vspace{0pt}
        \scalebox{0.17}{\includegraphics{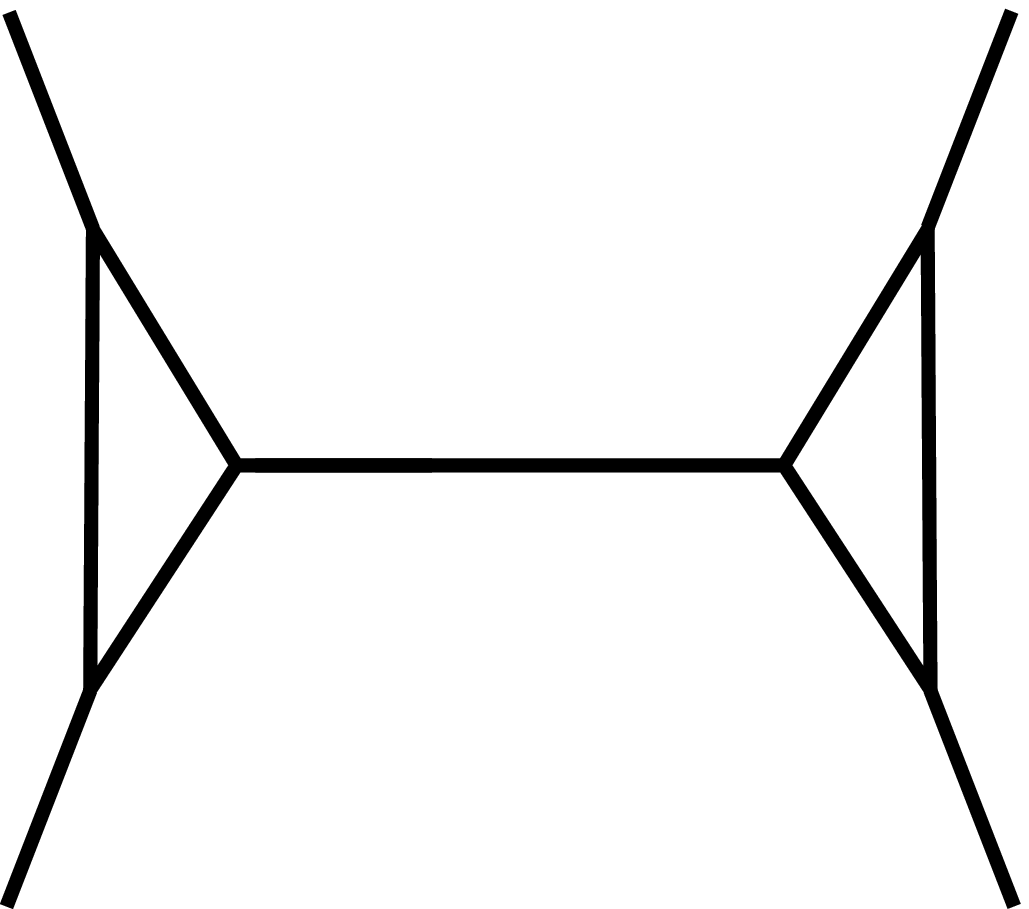}}
             \put(-29,+24){\footnotesize{$2i$}}
        \put(-60,+35){\footnotesize{$2n$}}
        \put(-60,-6){\footnotesize{$2n$}}
        \put(1,-6){\footnotesize{$2n$}}
        \put(1,+35){\footnotesize{$2n$}}
   \end{minipage}\hspace{9 mm}
   = \sum\limits_{i=0}^{n} B_{n,i} 
\hspace{1 mm}  
    \begin{minipage}[h]{0.09\linewidth}
        \vspace{0pt}
        \scalebox{0.36}{\includegraphics{horizantal_basis}}
             \put(-29,+17){\footnotesize{$2i$}}
        \put(-60,+26){\footnotesize{$2n$}}
        \put(-60,-6){\footnotesize{$2n$}}
        \put(1,-6){\footnotesize{$2n$}}
        \put(1,+26){\footnotesize{$2n$}}
   \end{minipage}
   \end{eqnarray*}
where

\begin{eqnarray}
B_{n,i}=\left(\frac{ Tet\left[ 
\begin{array}{ccc}
2i & n & n \\ 
n & 2n & 2n%
\end{array}%
\right] }{\theta (2n,2n,2i)}\right)^2\frac{\Delta_{2i}}{\theta (n,n,2i)}
\end{eqnarray}
But since,

\begin{eqnarray}
\label{2222212}
Tet\left[ 
\begin{array}{ccc}
2i & n & n \\ 
n & 2n & 2n%
\end{array}
\right] =\theta(2n,2n,2i).
\end{eqnarray}

One obtains:
\begin{equation*}
B_{n,i}=\frac{\Delta_{2i}}{\theta(n,n,2i)}
\end{equation*}
Furthermore,
\begin{eqnarray*}
    \begin{minipage}[h]{0.14\linewidth}
        \vspace{0pt}
        \scalebox{0.36}{\includegraphics{horizantal_basis}}
             \put(-29,+17){\footnotesize{$2i$}}
        \put(-60,+26){\footnotesize{$2n$}}
        \put(-60,-6){\footnotesize{$2n$}}
        \put(1,-6){\footnotesize{$2n$}}
        \put(1,+26){\footnotesize{$2n$}}
   \end{minipage}\otimes \hspace{5pt} \left( \hspace{10pt}
   \begin{minipage}[h]{0.11\linewidth}
        \vspace{0pt}
        \scalebox{0.22}{\includegraphics{bubble_as_spin}}
             \put(-29,+28){\footnotesize{$n$}}
             \put(-29,-5){\footnotesize{$n$}}
        \put(-60,+26){\footnotesize{$2n$}}
        \put(-60,-6){\footnotesize{$2n$}}
        \put(1,-6){\footnotesize{$2n$}}
        \put(1,+26){\footnotesize{$2n$}}
   \end{minipage} \hspace{10pt} \right)^{\otimes k}
   &=& (P_{n,i})^k
\hspace{1 mm}  
    \begin{minipage}[h]{0.09\linewidth}
        \vspace{0pt}
        \scalebox{0.36}{\includegraphics{horizantal_basis}}
             \put(-29,+17){\footnotesize{$2i$}}
        \put(-60,+26){\footnotesize{$2n$}}
        \put(-60,-6){\footnotesize{$2n$}}
        \put(1,-6){\footnotesize{$2n$}}
        \put(1,+26){\footnotesize{$2n$}}
   \end{minipage}
   \end{eqnarray*}
where 
\begin{eqnarray*}
P_{n,i}=\frac{  Tet\left[ 
\begin{array}{ccc}
2i & 2n & 2n \\ 
n & n & n%
\end{array}%
\right]  }{\theta (n,n,2i)}
\end{eqnarray*}
However,

\begin{eqnarray}
\label{22222}
Tet\left[ 
\begin{array}{ccc}
2i & n & n \\ 
n & 2n & 2n%
\end{array}
\right] =Tet\left[ 
\begin{array}{ccc}
2i & 2n & 2n \\ 
n & n & n%
\end{array}
\right]
\end{eqnarray}

Hence, both equations and (\ref{2222212}) and (\ref{22222}) imply:

\begin{eqnarray*}
P_{n,i}=\frac{ \theta(2n,2n,2i) }{\theta (n,n,2i)}
\end{eqnarray*}
Thus,

\begin{eqnarray*}
    \begin{minipage}[h]{0.34\linewidth}
        \vspace{0pt}
        \scalebox{0.36}{\includegraphics{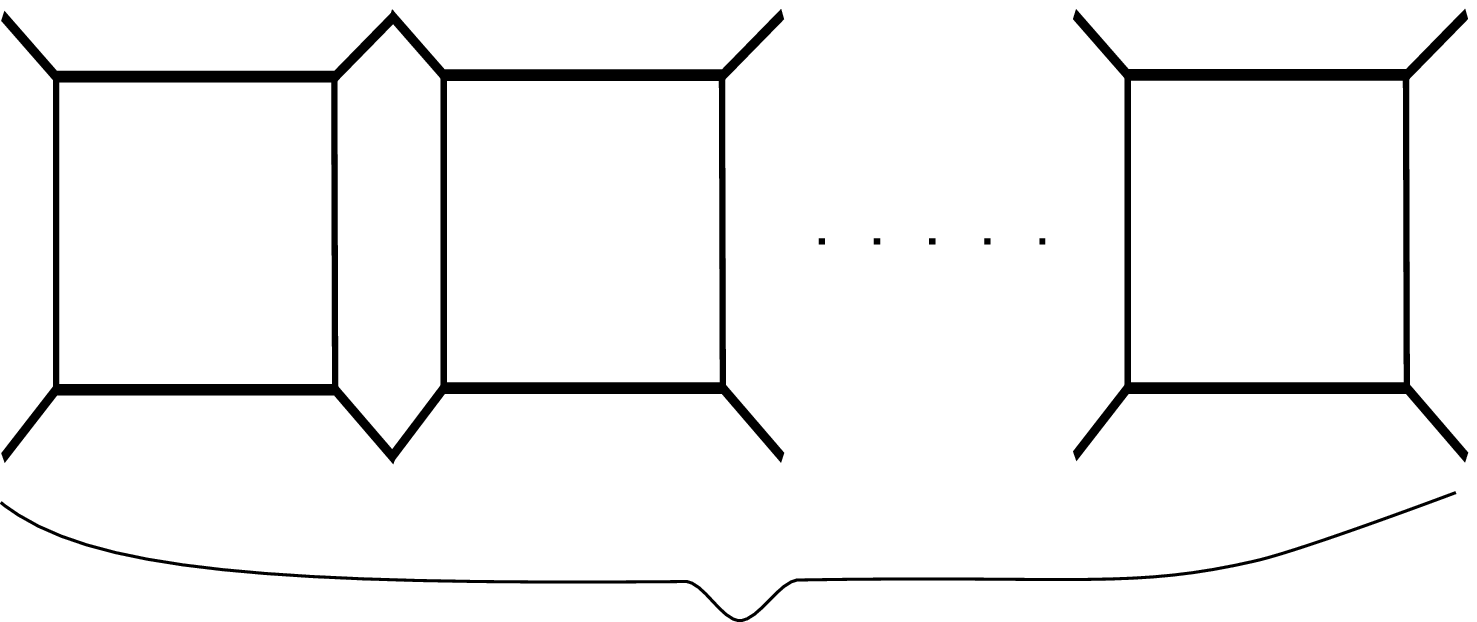}}
         \put(-29,+12){\footnotesize{$n$}}
        \put(-29,+60){\footnotesize{$n$}}
         \put(-160,64){\footnotesize{$2n$}}
        \put(-60,-6){\footnotesize{$k$ copies}}
        \put(-5,+64){\footnotesize{$2n$}}
        \put(-5,10){\footnotesize{$2n$}}
        \put(-160,+10){\footnotesize{$2n$}}
         \put(-139,+12){\footnotesize{$n$}}
        \put(-139,+60){\footnotesize{$n$}}
         \put(-99,+60){\footnotesize{$n$}}
        \put(-99,+12){\footnotesize{$n$}}
   \end{minipage}
   &=&\sum\limits_{i=0}^{n} B_{n,i}
\hspace{1 mm}  
   \begin{minipage}[h]{0.14\linewidth}
        \vspace{0pt}
        \scalebox{0.36}{\includegraphics{horizantal_basis}}
             \put(-29,+17){\footnotesize{$2i$}}
        \put(-60,+26){\footnotesize{$2n$}}
        \put(-60,-6){\footnotesize{$2n$}}
        \put(1,-6){\footnotesize{$2n$}}
        \put(1,+26){\footnotesize{$2n$}}
   \end{minipage}\otimes \hspace{5pt} \left( \hspace{10pt}
   \begin{minipage}[h]{0.11\linewidth}
        \vspace{0pt}
        \scalebox{0.22}{\includegraphics{bubble_as_spin}}
             \put(-29,+28){\footnotesize{$n$}}
             \put(-29,-5){\footnotesize{$n$}}
        \put(-60,+26){\footnotesize{$2n$}}
        \put(-60,-6){\footnotesize{$2n$}}
        \put(1,-6){\footnotesize{$2n$}}
        \put(1,+26){\footnotesize{$2n$}}
   \end{minipage} \hspace{10pt} \right)^{\otimes k-1}\\
   &=&\sum\limits_{i=0}^{n} B_{n,i} (P_{n,i})^{k-1}\hspace{1 mm}  
   \begin{minipage}[h]{0.14\linewidth}
        \vspace{0pt}
        \scalebox{0.36}{\includegraphics{horizantal_basis}}
             \put(-29,+17){\footnotesize{$2i$}}
        \put(-60,+26){\footnotesize{$2n$}}
        \put(-60,-6){\footnotesize{$2n$}}
        \put(1,-6){\footnotesize{$2n$}}
        \put(1,+26){\footnotesize{$2n$}}
   \end{minipage}.
   \end{eqnarray*}

Hence, \ref{small claim} follows. Thus colored Jones polynomial of $ST(k,l)$ is given by :

$$J_{2n+1,ST(k,l)}=\frac{1}{\Delta_{2n}} \sum_{i=0}^{n}\frac{\theta(2n,2n,2i)^{k}\Delta_{2i}}{\theta(n,n,2i)^{k}}(\lambda_{2i,2n})^l$$

\end{example}

\section{The tail of the Colored Jones Polynomial for Singular Knots}\label{sec5}
The Study of the properties of the tail of the colored Jones polynomial have attracted attention recently (see for instance \cite{Hajij1,Hajij2,lee2016trivial,robert1,milas}).  One of the main reasons for this is due to the fact that this tail have been proved to give rise to Ramanujan theta and false theta identities \cite{CodyOliver,Hajij2}.
In this section we start the investigation of the properties of the singular colored Jones polynomial and we compute the tail of the torus singular knot $ST_k$.  First, we briefly review the basics of the head and the tail of the colored Jones polynomial. For more details see \cite{Armond,CodyOliver,Hajij1, Hajij2}.
 
 \noindent
 If $P_1(q)$ and $P_2(q)$ are elements in $\mathbb{Z}[q^{-1}][[q]]$, we write $P_1(q)\doteq_n P_2(q)$ if their first $n$ coefficients agree up to a sign.  It was proven in \cite{CodyOliver} that the coefficients of the colored Jones polynomial of an alternating link $L$ stabilize in the following sense: For every $n\geq2$, we have $J_{n+1,L}(q)\doteq_n J_{n,L}(q)$. We give the following example to illustrate this further
\begin{example}
The colored Jones polynomial for the knot $6_2$, up to multiplication with a suitable power $q^{\pm a_n}$ for some integer $a_n$, is given in the following table:

\begin{center}
\vspace{0.5em}
\begin{small}
\begin{tabular}{ |c|c| } 

\hline

$n=2$ &\hspace{-4.29em} $1-2q+2q^{2}-2q^{3}+2q^{4}-q^{5}+q^{6}$  \\
$n=3$ & \hspace{-1.2em} $1-2q+4q^{3}-5q^{4}+6q^{6}-6q^{7}+6q^{9}+...$  \\ 
$n=4$ & \hspace{-1.7em} $1-2q+2q^{3}+q^{4}-4q^{5}-2q^{6}+7q^{7}+...$ \\ 
$n=5$ & \hspace{-1.35em}$1 - 2 q + 2 q^3 - q^4 + 2 q^5 - 6 q^6 + 2 q^7 +...$ \\ 
$n=6$ &  \hspace{-2.25em} $1-2q+2q^{3}-q^{4}-2q^{7}+q^{8}+5q^{9}+...$ \\ 
$n=7$ &  \hspace{1.6em}$1-2q+2q^{3}-q^{4}-2q^{6}+4q^{7}-3q^{8}+7q^{10}+...$ \\ 
$n=8$ &  \hspace{0.9em} $1-2q+2q^{3}-q^{4}-2q^{6}+2q^{7}+3q^{8}-4q^{9}+...$\\
\hline 
\end{tabular}
\end{small}
\end{center}
 This motivated the authors of \cite{CodyOliver} to define the tail of the colored Jones polynomial of a link. More precisely, define the $q$-series series associated with the colored Jones polynomial of an alternating link $L$ whose $n^{th}$ coefficient is the $n^{th}$ coefficient of $J_{n,L}(q)$. Stated differently, the tail of the colored Jones polynomial of a link $L$ is defined to be a series
 $T_L(q)$, that satisfies  $T_L(q)\doteq_{n}J_{n,L}(q)$ for all $n \geq 1$. Hence from the table above we deduce that the tail of the colored Jones polynomial of the knot $6_2$ is given by :
\begin{equation*}
T_{6_2}(q)=1-2q+ 0 q^{2} +2q^{3}-q^{4}+0q^{5}- 2q^{6}+2q^{7}+...
\end{equation*}
\end{example}
In the same way, the head of the colored Jones polynomial of a link $L$ is defined to be the tail of $J_{n,L}(q^{-1})$. In this paper we consider the tail of a sequence of invariants of singular knots. For this reason we define the tail of a sequence of power series in general.
\begin{definition}
Let $\mathcal{P} = \{P_n(q)\}_{n\in \mathbb{N}}$ be a sequence of formal power series in $\mathbb{Z}[q^{-1}][[	q]]$. The tail of the sequence $\mathcal{P}$- if it exists - is the formal power series $T_{\mathcal{P}}$ in $Z[[q]]$ that satisfies
$$T_{\mathcal{P}}(q)\doteq_n P_n(q)$$
\end{definition}
Next we prove and compute the tail of the colored Jones polynomial for the singular knot $ST_k$. This is the link shown in Figure \ref{Singular Torus knot} with only singular crossings.

\begin{theorem}
\label{first side}
For $k\geq 1$, we have
\begin{equation}
T_{ST_k}(q)= (q;q)_{\infty}^k \sum_{i=0}^\infty \frac{q^{i}}{(q;q)_i}
\end{equation}
\begin{proof}
The colored Jones polynomial of the singular knot $ST_k$ is given by:
\begin{equation}
\label{colored Jones}
J_{2n,ST_k}=\frac{1}{\Delta_{2n}}\sum_{i=0}^{n}\frac{\theta(2n,2n,2i)^{k}\Delta_{2i}}{\theta(n,n,2i)^{k}}.
\end{equation}

The theorem hence follows by proving that:

$$  J_{2n+1,ST_k}=_n  (q;q)_{\infty}^k \sum_{i=0}^\infty \frac{q^{i}}{(q;q)_i}.$$

First we note that equation \ref{pac} implies :
\begin{equation}
\label{special theta}
\frac{\theta(2n,2n,2i)}{\theta(n,n,2i)}=\frac{(-1)^n q^{\frac{-n}{2}} (q;q)_n^2 (q;q)_{2n-i} (q;q)_{2n+i+1} }{ (q;q)_{2n}^2 (q;q)_{n-i} (q;q)_{n+i+1} }
\end{equation}
On the other hand,

\begin{align}
\label{0}
  \frac{(q;q)_{n}} {(q;q)_{2n}}&=\frac{\displaystyle\prod_{k=0}^{n-1}(1-q^{k+1})}{\displaystyle\prod_{k=0}^{2n-1}(1-q^{k+1})} \nonumber \\
  &=\frac{1}{\displaystyle\prod_{k=n}^{2n-1}(1-q^{k+1})}\nonumber \\
  &=\displaystyle\prod_{k=0}^{n-1}\frac{1}{(1-q^{n+k+1})}\doteq_n1.
\end{align}

Moreover,

\begin{align}
\label{1}
\frac{(q;q)_{2n-i+1}}{(q;q)_{n+i+1}}
&=1-q^{n+i+2}+O(n+i+3)\doteq_{n}1.
\end{align}
and 

\begin{align}
\label{1}
\frac{(q;q)_{2n+i+1}}{(q;q)_{2n}}
&=1-q^{2n+1}+O(2n+2)\doteq_{n}1.
\end{align}

Hence equation \ref{colored Jones} becomes:

\begin{equation}
J_{2n+1,ST_k}\doteq_n \frac{(q;q)^k_n}{\Delta_n}\sum_{i=0}^n \frac{\Delta_{2i}}{(q;q)_{n-i}}=\frac{(q;q)^k_n}{\Delta_n}\sum_{i=0}^n \frac{\Delta_{2i}}{(q;q)_{i}} \doteq_n (q;q)^k_n\sum_{i=0}^n \frac{q^i}{(q;q)_{i}}
\end{equation}
The result follows.
\end{proof}
\end{theorem}

The special case when $k=2$ can be computed using another method that gives rise to an interesting false theta function identity. We do this by utilizing another method to evaluate the colored Jones polynomial of the singular knot $ST_{2}$. Using  Definition \ref{main definition} we see that $[ST_2]_{2n}$ is equal to the evaluation of following skein element :
\begin{eqnarray} 
\label{1111}
  [ST_2]_{2n} = \hspace{10pt}
   \begin{minipage}[h]{0.11\linewidth}
        \vspace{0pt}
        \scalebox{0.22}{\includegraphics{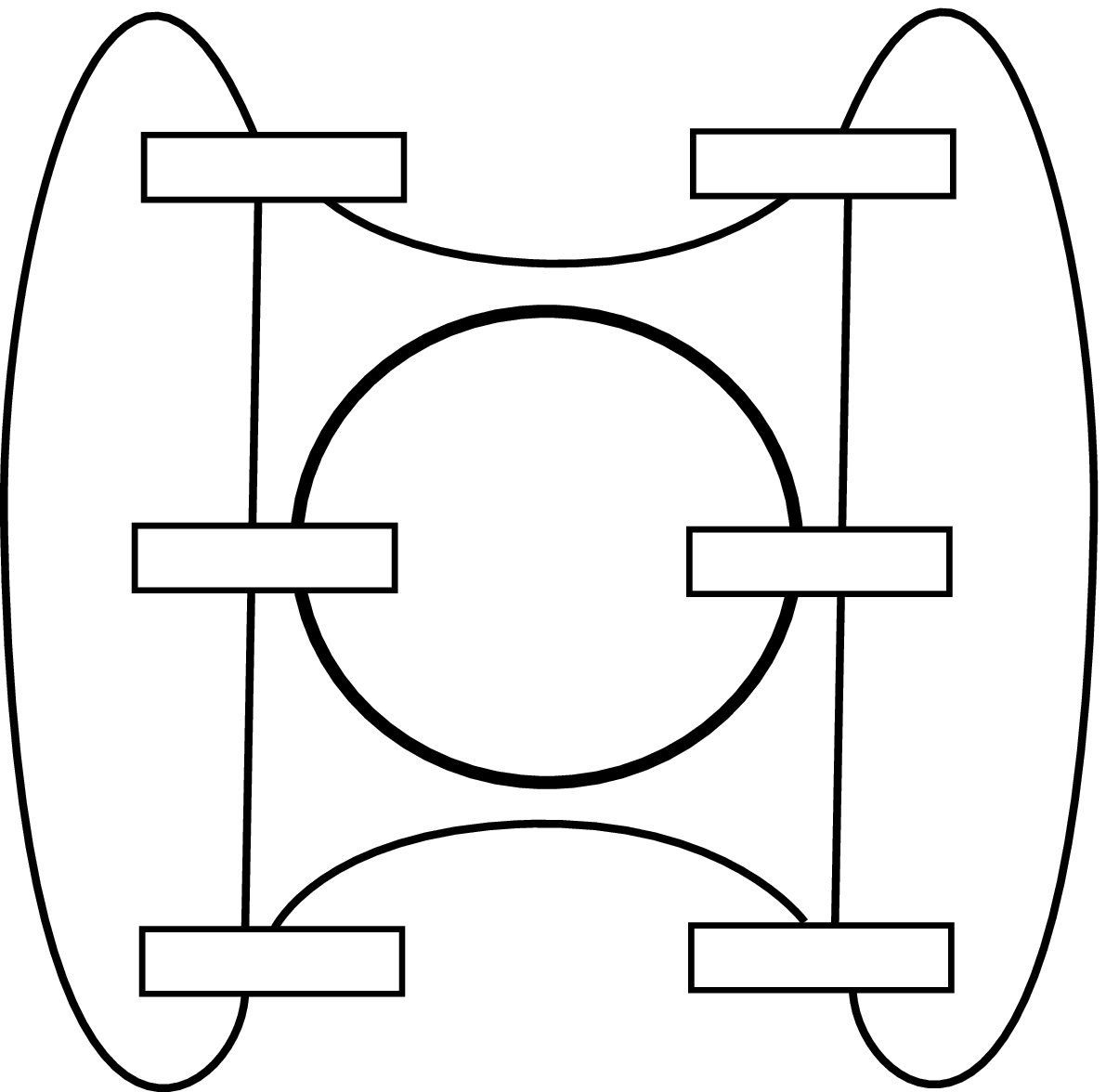}}
        \put(-38,+28){\footnotesize{$n$}}
        \put(-87,+45){\footnotesize{$2n$}}
        \put(5,+22){\footnotesize{$2n$}}
   \end{minipage} 
   \end{eqnarray}
This can be used to show the following result.

\begin{theorem}
\label{second side}
\begin{equation}
\label{1}
T_{ST_2}\doteq_n (q;q)_n \sum_{i=0}^n \frac{q^{i^2+i}}{(q;q)^2_i}
\end{equation}
\end{theorem}
\begin{proof}
We use the bubble skein formula given in \cite{Hajij1} on the bubble showing in skein element on the right hand side of  equation (\ref{1111}), we obtain 

\begin{eqnarray}  
  [ST_2]_{2n} &=& \displaystyle\sum\limits_{i=0}^{n}
   \left\lceil 
\begin{array}{cc}
n & n \\ 
n & n%
\end{array}%
\right\rceil _{i} \hspace{10pt}
   \begin{minipage}[h]{0.11\linewidth}
        \vspace{0pt}
        \scalebox{0.22}{\includegraphics{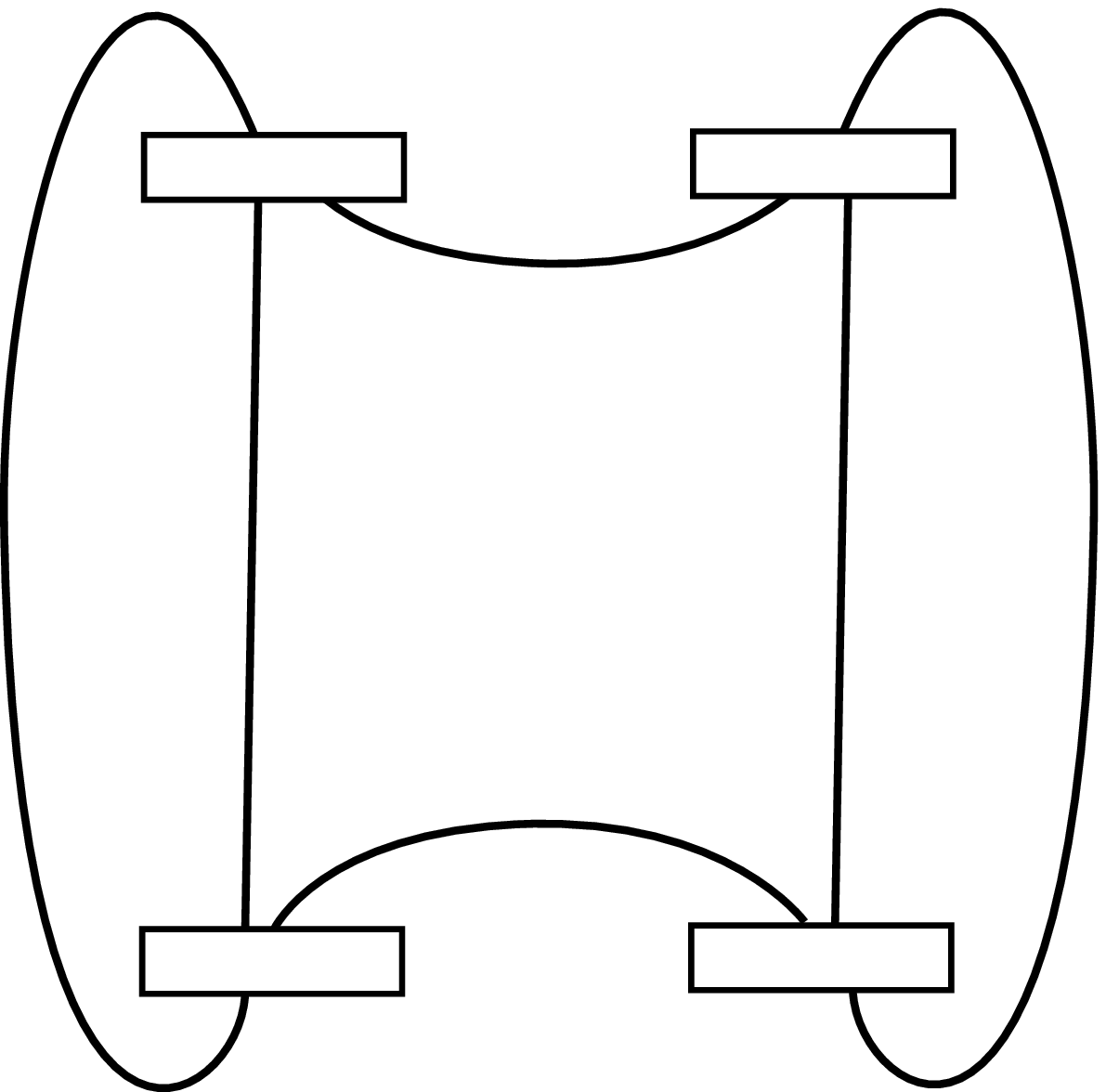}}
        \put(-49,+25){\footnotesize{$n+i$}}
         \put(-49,+45){\footnotesize{$n+i$}}
        \put(-85,+55){\footnotesize{$2n$}}
        \put(5,+22){\footnotesize{$2n$}}
   \end{minipage}
   \\&=& \displaystyle\sum\limits_{i=0}^{n}
   \left\lceil 
\begin{array}{cc}
n & n \\ 
n & n%
\end{array}%
\right\rceil _{i} \frac{\Delta_{2n}^2}{\Delta_{n+i}}.
   \end{eqnarray}
For the definition of the coefficient $ \left\lceil 
\begin{array}{cc}
n & n \\ 
n & n%
\end{array}%
\right\rceil _{i}$ see \cite{Hajij2} Theorem $2.4$. Hence,
\begin{equation}
J_{2n,ST_k}=\displaystyle\sum\limits_{i=0}^{n}
   \left\lceil 
\begin{array}{cc}
n & n \\ 
n & n%
\end{array}%
\right\rceil _{i} \frac{\Delta_{2n}}{\Delta_{n+i}}
\end{equation}

By Lemma 4.10 part (1) of \cite{Hajij2} we have

\begin{equation}
\displaystyle\sum\limits_{i=0}^{n}
   \left\lceil 
\begin{array}{cc}
n & n \\ 
n & n%
\end{array}%
\right\rceil _{i} \frac{\Delta_{2n}}{\Delta_{n+i}}\doteq_n (q;q)_n \sum_{i=0}^n \frac{q^{i^2+i}}{(q;q)^2_i}
\end{equation}
The result follows.
\end{proof}
Theorems \ref{first side} and \ref{second side} imply immediately the following:

\begin{corollary}
\label{identity}
The following identity holds : 
\begin{equation}
 (q;q)_{\infty} \sum_{i=0}^{\infty} \frac{q^{i^2+i}}{(q;q)^2_i}=(q;q)^2_{\infty}\sum_{i=0}^{\infty} \frac{q^i}{(q;q)_{i}}
\end{equation}
\end{corollary}
 The identity in Corollary \ref{identity} is a well-known false theta function identity. In fact both sides are equal to $\Psi(q^3,q)$ where, 
 \begin{equation}
\Psi(a,b)=\sum\limits_{i=0}^{\infty}a^{\frac{i(i+1)}{2}}b^{\frac{i(i-1)}{2}}-\sum\limits_{i=1}^{\infty}a^{\frac{i(i-1)}{2}}b^{\frac{i(i+1)}{2}} 
 \end{equation}

 See for instance page 169 of \cite{Andrews} or see page 200 in \cite{Ramanujan}. For a recent study of the previous identity, also related to the colored Jones polynomial, see also the work of Bringmann and Milas in \cite{milas}.

\section{Conclusion}

The existence of the tail of the colored Jones polynomial of singular links is still an open question. The tail of the colored Jones polynomial of non-singular links exists for adequate links. The question of an analogue of adequate links in the singular case seems to be an interesting question that is worth pursuing.

\subsection{Acknowledgment}
We would like to thank Antun Milas for useful conversations and suggesting a correction in this paper.


\begin{thebibliography}{15}


\bib{Andrews}{book}{
   author={Andrews, George E.},
   author={Berndt, Bruce C.},
   title={Ramanujan's lost notebook. Part I},
   publisher={Springer, New York},
   date={2005},
   pages={xiv+437},
   isbn={978-0387-25529-3},
   isbn={0-387-25529-X},
   review={\MR{2135178 }},
}

\bib{Armond}{article}{
   author={Armond, Cody},
   title={The head and tail conjecture for alternating knots},
   journal={Algebr. Geom. Topol.},
   volume={13},
   date={2013},
   number={5},
   pages={2809--2826},
   issn={1472-2747},
   review={\MR{3116304}},
   doi={10.2140/agt.2013.13.2809},
}

\bib{CodyOliver}{article}{
   author={Armond, Cody},
   author={Dasbach, Oliver},
   title={Rogers-Ramanujan type identities and the head and tail of the colored Jones polynomial},
   journal={arXiv:1106.3948 },
   date={2011},
}




\bib{bataineh2016colored}{article}{
   author={Bataineh, Khaled},
   author={Elhamdadi, Mohamed},
   author={Hajij, Mustafa},
   title={The colored Jones polynomial of singular knots},
   journal={New York J. Math.},
   volume={22},
   date={2016},
   pages={1439--1456},
   issn={1076-9803},
   review={\MR{3603072}},
}

\bib{BHMV}{article}{
   author={Blanchet, C.},
   author={Habegger, N.},
   author={Masbaum, G.},
   author={Vogel, P.},
   title={Topological quantum field theories derived from the Kauffman
   bracket},
   journal={Topology},
   volume={34},
   date={1995},
   number={4},
   pages={883--927},
   issn={0040-9383},
   review={\MR{1362791 (96i:57015)}},
}

\bib{milas}{article}{
  title={-Algebras, False Theta Functions and Quantum Modular Forms, I},
  author={Bringmann, Kathrin and Milas, Antun},
  journal={International Mathematics Research Notices},
  pages={rnv033},
  year={2015},
  publisher={Oxford University Press}
}
\bib{DL}{article}{
   author={Dasbach, Oliver T.},
   author={Lin, Xiao-Song},
   title={On the head and the tail of the colored Jones polynomial},
   journal={Compos. Math.},
   volume={142},
   date={2006},
   number={5},
   pages={1332--1342},
   issn={0010-437X},
   review={\MR{2264669}},
   doi={10.1112/S0010437X06002296},
}


\bib{singularquandle}{article}{ title={Singular Knots and Involutive Quandles},
  author={Churchill, Indu Rasika}
  author={Elhamdadi, Mohamed}
  author={Hajij, Mustafa}
  author={Nelson, Sam},
  journal={To Appear in Journal of Knot Theory and Its Ramifications},
  year={2017}
}

\bib{singularquandle1}{article}{
  title={Generating sets of Reidemeister moves of oriented singular links and quandles},
  author={Bataineh, Khaled}
  author={ Elhamdadi, Mohamed}
  author={Hajij, Mustafa}
  author={ Youmans, William}
  journal={arXiv preprint arXiv:1702.01150},
  year={2017}
}


\bib{EH}{article}{
	author={Elhamdadi, Mohamed},
	author={Hajij, Mustafa},
	title={Pretzel Knots and Q-Series},
	journal={Osaka J. Math. }
	volume={54},
	date={2017},
	number={2},
	pages={363-381},
	issn={},
}

\bib{EH}{article}{
	author={Elhamdadi, Mohamed},
	author={Hajij, Mustafa},
    author={Saito, Masahico},
	title={Twist Regions and Coefficients Stability of the Colored Jones Polynomial },
	journal={Trans. Amer. Math. Soc. }
	volume={},
	date={2017},
	number={},
	issn={},
	doi={https://doi.org/10.1090/tran/7128},
}


\bib{Fiedler}{article}{
	author={Fiedler, Thomas},
	title={The Jones and Alexander polynomials for singular links},
	journal={J. Knot Theory Ramifications},
	volume={19},
	date={2010},
	number={7},
	pages={859--866},
	issn={0218-2165},
	review={\MR{2673687 (2012b:57024)}},
}

\bib{GL}{article}{
   author={Garoufalidis, Stavros},
   author={L\^e, Thang T. Q.},
   title={Nahm sums, stability and the colored Jones polynomial},
   journal={Res. Math. Sci.},
   volume={2},
   date={2015},
   pages={Art. 1, 55},
   issn={2197-9847},
   review={\MR{3375651}},
   doi={10.1186/2197-9847-2-1},
}
		
\bib{Gemein}{article}{
	author={Gemein, Bernd},
	title={Singular braids and Markov's theorem},
	journal={J. Knot Theory Ramifications},
	volume={6},
	date={1997},
	number={4},
	pages={441--454},
	issn={0218-2165},
	review={\MR{1466593 (98i:57008)}},
}

\bib{Hajij3}{article}{
   title={The colored Kauffman skein relation and the head and tail of the colored jones polynomial},
  author={Hajij, Mustafa},
  journal={Journal of Knot Theory and Its Ramifications},
  volume={26},
  number={03},
  pages={1741002},
  year={2017},
  publisher={World Scientific}
}


\bib{Hajij2}{article}{
   author={Hajij, Mustafa},
   title={The tail of a quantum spin network},
   journal={Ramanujan J.  },
   volume={Vol 40},
   date={2016},
   number={1},
   pages={pp 135-176},
   issn={},
   review={},
}

\bib{Hajij1}{article}{
   author={Hajij, Mustafa},
   title={The Bubble skein element and applications},
   journal={J. Knot Theory Ramifications},
   volume={23},
   date={2014},
   number={14},
   pages={1450076, 30},
   issn={0218-2165},
   review={\MR{3312619}},
}

\bib{Hikami}{article}{
   author={Hikami, Kazuhiro},
   title={Volume conjecture and asymptotic expansion of $q$-series},
   journal={Experiment. Math.},
   volume={12},
   date={2003},
   number={3},
   pages={319--337},
   issn={1058-6458},
   review={\MR{2034396}},
}


\bib{lee2016trivial}{article}{
  title={A trivial tail homology for non $ A $-adequate links},
  author={Lee, Christine Ruey Shan},
  journal={arXiv preprint arXiv:1611.00686},
  year={2016}
}

\bib{lee2016trivial}{article}{
  title={A trivial tail homology for non $ A $-adequate links},
  author={Lee, Christine Ruey Shan},
  journal={arXiv preprint arXiv:1611.00686},
  year={2016}
}

\bib{J2}{article}{
	author={Jones, Vaughan F. R.},
	title={Index for subfactors},
	journal={Invent. Math.},
	volume={72},
	date={1983},
	number={1},
	pages={1--25},
	issn={0020-9910},
	review={\MR{696688 (84d:46097)}},
	doi={10.1007/BF01389127},
}

\bib{JL}{article}{
	author={Juyumaya, J.},
	author={Lambropoulou, S.},
	title={An invariant for singular knots},
	journal={J. Knot Theory Ramifications},
	volume={18},
	date={2009},
	number={6},
	pages={825--840},
	issn={0218-2165},
	review={\MR{2542698 (2010i:57028)}},
}

\bib{kaufgraph}{article}{
   author={Kauffman, Louis H.},
   title={Invariants of graphs in three-space},
   journal={Trans. Amer. Math. Soc.},
   volume={311},
   date={1989},
   number={2},
   pages={697--710},
   issn={0002-9947},
   review={\MR{946218 (89f:57007)}},
   doi={10.2307/2001147},
}


\bib{Kauffman}{article}{
   author={Kauffman, Louis H.},
   title={State models and the Jones polynomial},
   journal={Topology},
   volume={26},
   date={1987},
   number={3},
   pages={395--407},
   issn={0040-9383},
   review={\MR{899057 }},
}

\bib{KauffLin}{book}{
   author={Kauffman, Louis H.},
   author={Lins, S{\'o}stenes L.},
   title={Temperley-Lieb recoupling theory and invariants of $3$-manifolds},
   series={Annals of Mathematics Studies},
   volume={134},
   publisher={Princeton University Press, Princeton, NJ},
   date={1994},
   pages={x+296},
   isbn={0-691-03640-3},
   review={\MR{1280463 (95c:57027)}},
}
\bib{kaufvog}{article}{
	author={Kauffman, Louis H.},
	author={Vogel, Pierre},
	title={Link polynomials and a graphical calculus},
	journal={J. Knot Theory Ramifications},
	volume={1},
	date={1992},
	number={1},
	pages={59--104},
	issn={0218-2165},
	review={\MR{1155094 (92m:57012)}},
	doi={10.1142/S0218216592000069},
}

\bib{Robert}{article}{
   author={Keilthy, Adam},
   author={Osburn, Robert},
   title={Rogers-Ramanujan type identities for alternating knots},
   journal={J. Number Theory},
   volume={161},
   date={2016},
   pages={255--280},
   issn={0022-314X},
   review={\MR{3435728}},
   doi={10.1016/j.jnt.2015.02.002},
}
		
\bib{Lic92}{article}{
   author={Lickorish, W. B. R.},
   title={Calculations with the Temperley-Lieb algebra},
   journal={Comment. Math. Helv.},
   volume={67},
   date={1992},
   number={4},
   pages={571--591},
   issn={0010-2571},
   review={\MR{1185809 }},
}

\bib{Masbaum}{article}{
	author={Masbaum, Gregor},
	title={Skein-theoretical derivation of some formulas of Habiro},
	journal={Algebr. Geom. Topol.},
	volume={3},
	date={2003},
	pages={537--556 (electronic)},
	issn={1472-2747},
	review={\MR{1997328 (2004f:57013)}},
	doi={10.2140/agt.2003.3.537},
}

\bib{MV}{article}{
	author={Masbaum, G.},
	author={Vogel, P.},
	title={$3$-valent graphs and the Kauffman bracket},
	journal={Pacific J. Math.},
	volume={164},
	date={1994},
	number={2},
	pages={361--381},
	issn={0030-8730},
	review={\MR{1272656 }},
}
	
\bib{Morton}{article}{
	author={Morton, H. R.},
	title={Invariants of links and $3$-manifolds from skein theory and from
		quantum groups},
	conference={
		title={Topics in knot theory},
		address={Erzurum},
		date={1992},
	},
	book={
		series={NATO Adv. Sci. Inst. Ser. C Math. Phys. Sci.},
		volume={399},
		publisher={Kluwer Acad. Publ., Dordrecht},
	},
	date={1993},
	pages={107--155},
	review={\MR{1257908}},
}	


\bib{Paris}{article}{
	author={Paris, Luis},
	title={The proof of Birman's conjecture on singular braid monoids},
	journal={Geom. Topol.},
	volume={8},
	date={2004},
	pages={1281--1300 (electronic)},
	issn={1465-3060},
	review={\MR{2087084}},
}
		
\bib{Przytycki}{article}{
   author={Przytycki, J{\'o}zef H.},
   title={Fundamentals of Kauffman bracket skein modules},
   journal={Kobe J. Math.},
   volume={16},
   date={1999},
   number={1},
   pages={45--66},
   issn={0289-9051},
   review={\MR{1723531 }},
}
\bib{PT}{article}{
   author={Przytycki, J{\'o}zef H.},
   author={Traczyk, Pawe{\l}},
   title={Invariants of links of Conway type},
   journal={Kobe J. Math.},
   volume={4},
   date={1988},
   number={2},
   pages={115--139},
   issn={0289-9051},
   review={\MR{945888 }},
}

\bib{robert1}{article}{
  title={q-series and tails of colored Jones polynomials},
  author={Beirne, Paul and Osburn, Robert},
  journal={Indagationes Mathematicae},
  volume={28},
  number={1},
  pages={247--260},
  year={2017},
  publisher={Elsevier}
}

\bib{ReshetikhinTuraevRibbon}{article}{
   author={Reshetikhin, N. Yu.},
   author={Turaev, V. G.},
   title={Ribbon graphs and their invariants derived from quantum groups},
   journal={Comm. Math. Phys.},
   volume={127},
   date={1990},
   number={1},
   pages={1--26},
   issn={0010-3616},
   review={\MR{1036112 }},
}

\bib{Ramanujan}{book}{
   author={Ramanujan, Srinivasa},
   title={The lost notebook and other unpublished papers},
   note={With an introduction by George E. Andrews},
   publisher={Springer-Verlag, Berlin; Narosa Publishing House, New Delhi},
   date={1988},
   pages={xxviii+419},
   isbn={3-540-18726-X},
   review={\MR{947735}},
}


\bib{rozansky2012khovanov}{article}{
   author={Rozansky, Lev},
   title={Khovanov homology of a unicolored B-adequate link has a tail},
   journal={Quantum Topol.},
   volume={5},
   date={2014},
   number={4},
   pages={541--579},
   issn={1663-487X},
   review={\MR{3317343}},
   doi={10.4171/QT/58},
}
	


\bib{RT}{article}{
   author={Reshetikhin, N.},
   author={Turaev, V. G.},
   title={Invariants of $3$-manifolds via link polynomials and quantum
   groups},
   journal={Invent. Math.},
   volume={103},
   date={1991},
   number={3},
   pages={547--597},
   issn={0020-9910},
   review={\MR{1091619 }},
}

\bib{TemperleyLieb}{article}{
	author={Temperley, H. N. V.},
	author={Lieb, E. H.},
	title={Relations between the ``percolation'' and ``colouring'' problem
		and other graph-theoretical problems associated with regular planar
		lattices: some exact results for the ``percolation'' problem},
	journal={Proc. Roy. Soc. London Ser. A},
	volume={322},
	date={1971},
	number={1549},
	pages={251--280},
	issn={0962-8444},
	review={\MR{0498284 (58 \#16425)}},
}

\bib{TuraevViro92}{article}{
   author={Turaev, V. G.},
   author={Viro, O. Ya.},
   title={State sum invariants of $3$-manifolds and quantum $6j$-symbols},
   journal={Topology},
   volume={31},
   date={1992},
   number={4},
   pages={865--902},
   issn={0040-9383},
   review={\MR{1191386 }},
}

\bib{TuraevWenzl93}{article}{
   author={Turaev, V.},
   author={Wenzl, H.},
   title={Quantum invariants of $3$-manifolds associated with classical
   simple Lie algebras},
   journal={Internat. J. Math.},
   volume={4},
   date={1993},
   number={2},
   pages={323--358},
   issn={0129-167X},
   review={\MR{1217386 }},
}

\bib{Vassiliev}{article}{
	author={Vassiliev, V. A.},
	title={Cohomology of knot spaces},
	conference={
		title={Theory of singularities and its applications},
	},
	book={
		series={Adv. Soviet Math.},
		volume={1},
		publisher={Amer. Math. Soc., Providence, RI},
	},
	date={1990},
	pages={23--69},
	review={\MR{1089670 (92a:57016)}},
}


\bib{Wenzl}{article}{
   author={Wenzl, Hans},
   title={On sequences of projections},
   journal={C. R. Math. Rep. Acad. Sci. Canada},
   volume={9},
   date={1987},
   number={1},
   pages={5--9},
   issn={0706-1994},
   review={\MR{873400 }},
}



	

		
\end{thebibliography}
\end{document}